\documentclass{amsart}
\usepackage{latexsym,amsxtra,amscd,ifthen}
\usepackage{amsfonts}
\usepackage{verbatim}
\usepackage{amsmath}
\usepackage{amsthm}
\usepackage{amssymb}

\numberwithin{equation}{section}

\theoremstyle{plain}
\newtheorem{theorem}{Theorem}[section]
\newtheorem{lemma}[theorem]{Lemma}
\newtheorem{proposition}[theorem]{Proposition}

\newcommand{\inth}{\textstyle \int}
\theoremstyle{definition}
\newtheorem{definition}[theorem]{Definition}

\newtheorem{example}[theorem]{Example}

\newtheorem{remark}[theorem]{Remark}

\makeatletter              
\let\c@equation\c@theorem  
\makeatother

\newcommand{\rank}{\operatorname{rank}}

\DeclareMathOperator{\hdet}{hdet}

\DeclareMathOperator{\Ext}{Ext} \DeclareMathOperator{\Tor}{Tor}

\DeclareMathOperator{\gr}{gr}

\DeclareMathOperator{\Aut}{Aut}

\DeclareMathOperator{\AutSL}{SLA}

\DeclareMathOperator{\GKdim}{GKdim}

\newcommand{\cal}{\mathcal}

\begin{document}

\title[Invariant theory on down-up algebras]
{Invariant theory of finite group actions \\
on down-up algebras}

\author{E. Kirkman, J. Kuzmanovich and J.J. Zhang}

\address{Kirkman: Department of Mathematics,
P. O. Box 7388, Wake Forest University, Winston-Salem, NC 27109}

\email{kirkman@wfu.edu}

\address{Kuzmanovich: Department of Mathematics,
P. O. Box 7388, Wake Forest University, Winston-Salem, NC 27109}

\email{kuz@wfu.edu}

\address{zhang: Department of Mathematics, Box 354350,
University of Washington, Seattle, Washington 98195, USA}

\email{zhang@math.washington.edu}

\begin{abstract}
We study Artin-Schelter
Gorenstein fixed subrings of some Artin-Schelter regular
algebras of dimension 2 and 3 under a finite group action
and prove a noncommutative version of the Kac-Watanabe and Gordeev
theorem for these algebras.
\end{abstract}







\maketitle

\setcounter{section}{-1}
\section{introduction}
\label{zzsec0}

This paper continues our project of extending classical invariant 
theory to a noncommutative context.  Throughout, let $A$ be an 
Artin-Schelter regular (AS regular, for short) algebra over an 
algebraically closed field $k$ of characteristic zero, let $\Aut(A)$ 
be the group of graded automorphisms of $A$, and let $H$ be a 
finite subgroup of $\Aut(A)$. Several fundamental results in 
classical invariant theory have been
generalized to this noncommutative setting; for example, see
\cite[Theorem 1.1]{KKZ3} for a version of the Shephard-Todd-Chevalley
Theorem, and \cite[Theorem 0.2]{KKZ2} for a version of Watanabe's
theorem.

Our interest here is in studying AS Gorenstein fixed rings $A^H$ of 
$A$.  We have chosen to study the parametrized family $A(\alpha, \beta)$ 
of $\mathbb{N}$-graded down-up algebras.  Down-up algebras were 
introduced by Benkart and Roby \cite{BR} in 1998 as a tool to study 
the structure of certain posets. Since then various aspects of these 
algebras have been studied extensively, for example, \cite{BW, KK, KMP}. 
Noetherian graded down-up algebras form a class of AS regular algebras 
of global dimension three that are generated in degree 1 by two 
elements, with two cubic relations; these algebras are not Koszul, but 
they are (3)-Koszul.  Their graded automorphism groups, which depend 
upon the parameters $\alpha$ and $\beta$, were computed in \cite{KK}, 
and are sufficiently rich to provide many non-trivial examples (e.g. 
in two cases the automorphism group is the entire group $GL_2(k)$).

To accomplish our study of down-up algebras, we first study AS 
Gorenstein fixed subrings of AS regular algebras of global dimension two. 
In Section \ref{zzsec2} we prove the following extension of a 
commutative result \cite[Proposition 9.4]{St2} by essentially 
classifying all AS Gorenstein fixed subrings of AS regular algebras of 
dimension two.

\begin{theorem}
\label{zzthm0.1}
Let $A$ be a noetherian AS regular algebra of global dimension two
that is generated in degree one. Let $H$ be a finite subgroup of
$\Aut(A)$. Then the following are equivalent.
\begin{enumerate}
\item
$A^H$ is AS Gorenstein.
\item
$A^H$ is isomorphic to $C/(\Omega)$, where $C$ is a noetherian AS
regular algebra of global dimension three, and $\Omega$ is a
regular normal element of $C$. {\rm{(}}Such an algebra $C/(\Omega)$ is
called a hypersurface.{\rm{)}}
\end{enumerate}
\end{theorem}

For down-up algebras, two fundamental results of classical invariant theory,
the Shephard-Todd-Chevalley Theorem and Watanabe's Theorem, hold in the
following way.

\begin{proposition}
\label{zzpro0.2} Let $A$ be a noetherian graded down-up algebra
and let $\{1\}\neq H$ be a finite subgroup of $\Aut(A)$.
\begin{enumerate}
\item \cite[Proposition 6.4]{KKZ1}
$A$ has no quasi-reflections of finite order. As a consequence,
$A^H$ is not AS regular, and all finite groups of graded automorphisms
are ``small".
\item
\cite[Corollary 4.11]{KKZ2}
$A^H$ is AS Gorenstein if and only if the homological determinant
of the $H$-action is trivial.
\end{enumerate}
\end{proposition}

In studying AS Gorenstein invariant subrings we are interested in a 
noncommutative version of the classical
result of Kac-Watanabe \cite{KW} and Gordeev \cite{G1} that states that,
if the fixed subring $k[x_1,\cdots,x_n]^H$ (for any finite subgroup
$H \subset GL_n(k)$) is a complete intersection, then $H$ is
generated by bireflections (i.e., elements $g\in GL_n(k)$ such that
$\rank (g-I)\leq 2$). In Section \ref{zzsec5} we will provide 
suitable definitions of complete intersection and bireflection in 
the noncommutative setting. In dimension 2, by Theorem \ref{zzthm0.1} 
all AS Gorenstein invariant subrings are hypersurfaces, and all
automorphisms of finite order are trivially bireflections, and hence 
the first interesting case of the Kac-Watanabe/Gordeev theorem is in 
dimension 3, so that it is natural to investigate generalizations of 
this theorem for down-up algebras. We prove the following version of 
Kac-Watanabe/Gordeev theorem for noetherian graded down-up algebras.

\begin{theorem}
\label{zzthm0.3} Let $A$ be a connected graded noetherian down-up
algebra and $H$ be a finite subgroup of $\Aut(A)$. Then the
following are equivalent.
\begin{enumerate}
\item[(C1)]
$A^H$ is a complete intersection {\rm{(}}of type GK{\rm{)}}.
\item[(C2)]
$A^H$ is cyclotomic Gorenstein and $H$ is generated by
bireflections.
\item[(C3)]
$A^H$ is cyclotomic Gorenstein.
\end{enumerate}
\end{theorem}

In general, it would be interesting to understand the
connection between three very different concepts
$$\text{``complete intersection'', \; ``cyclotomic
Gorenstein'' \; and \; ``bireflection''}$$ for a general
noncommutative AS regular algebra. It was proved in \cite{KKZ4} that
condition (C1) implies condition (C3) for $A^H$, where $A$
is any noetherian AS regular algebra, so it is a natural question
to ask if condition (C1) is equivalent to (C3) for algebras of the
form $A^H$, where $A$ is a noetherian AS regular algebra.
 In \cite{St1} R. Stanley asked when a commutative cyclotomic Gorenstein
ring is a complete intersection, and he gave an example that shows 
that this is not always the case \cite[Example 3.9]{St1}. The answer 
to Stanley's question is not yet well understood, even for cyclic 
groups, even when $A$ is a commutative polynomial ring  (see \cite{G3}). 
If $A$ is a commutative polynomial ring, then condition (C1) implies
(C2) by Kac-Watanabe/Gordeev \cite{KW, G1}. Another natural
question, which is a motivating question for this paper, is whether
the Kac-Watanabe/Gordeev theorem holds for general noncommutative
AS regular algebras.  In some cases, the fixed subrings that
appear in Theorem \ref{zzthm0.3} are presented explicitly
as (classical) complete intersections [Examples \ref{zzex8.2}
and \ref{zzex8.3}].

To prove Theorem \ref{zzthm0.3}, we first find the finite groups $H$ 
that act on some graded down-up algebra  $A:=A(\alpha, \beta)$ with 
trivial homological determinant; 
the set of all graded automorphisms with trivial homological determinant
is denoted by $SLA(A)$, as they can be considered 
as analogs of the elements in the special linear group $SL_n(k)$. 
In fact, for two families of down-up algebras all finite subgroups 
$SL_2(k)$ appear in the classification, but as the ``homological 
determinant of $g$" for $g$ a graded automorphism of a down-up 
algebra is $(\text{det } g)^2$, there are additional subgroups of 
$GL_2(k)$ in $SLA(A)$.  For all values of $\alpha$ and $\beta$, there 
are diagonal automorphisms of $A$ of the form
$$g = \begin{pmatrix}
a & 0 \\
0 & b
\end{pmatrix},$$
and we let $O := k^\times \times k^\times$ denote the group of all 
elements of that form, where $k^\times$ denotes $k\setminus \{0\}$. 
For some values of the parameters, these are the only kinds of graded 
automorphisms of $A(\alpha, \beta)$. For the parameterized family 
$A= A(\alpha, -1)$, there are groups of automorphisms of $A$ whose 
elements are either in $O$, or are of the form
$$g = \begin{pmatrix}
0 & a \\
b & 0
\end{pmatrix},$$
and we let $U$ denote the group of all elements that are of the form 
above, or in $O$.  We show that the families of finite subgroups of 
$O$ and $U$ that occur in $SLA(A)$, arise as involutions of type ADE 
singularities of $k[x,y]$. These families of finite groups are
classified in Section \ref{zzsec3}, and the eight families 
$Q_i, i= 1, \dots, 8$, that arise appear in Table 3 at the end of 
that section.  In Section \ref{zzsec4} we find the corresponding 
finite subgroups of $O$ for $k_q[x,y]$.

Our goal is to determine which of the AS Gorenstein subrings $A^H$ 
are ``complete intersections", but even for groups of small order 
it is difficult to obtain an explicit presentation of $A^H$.
Hence, an important ingredient in the proof of Theorem \ref{zzthm0.3} 
is lifting the various properties in (C1-C3) from the associated 
graded ring $\gr_F A$ to $A$, where $F$ is a
filtration of $A$ (definitions are given in Section \ref{zzsec6}).
For down-up algebras we often are able to choose a filtration so 
that $\gr_F A$ is an algebra where the computation is much easier.
The following results are proved in Section \ref{zzsec6}.

\begin{proposition}
\label{zzpro0.4} Let $A$ be a connected graded AS regular algebra and $H$ a
finite subgroup of $\Aut(A)$. Suppose $F$ is a graded $H$-stable
filtration of $A$ such that the associated Rees ring is noetherian
and that the associated graded algebra $\gr_F A$ is an AS 
Cohen-Macaulay algebra.
Let $\widehat{H}$ be the subgroup of $\Aut(\gr_F A)$ induced from $H$.
Then the following properties hold.
\begin{enumerate}
\item
If the fixed subring $(\gr_F A)^{\widehat{H}}$ is a complete
intersection of type GK, then so is $A^H$.
\item
Assume $\gr_F A$ is a domain. The fixed subring $(\gr_F A)^{\widehat{H}}$
is cyclotomic Gorenstein if and only if $A^H$ is.
\item
The group $\widehat{H}$ is generated by bireflections if and only
if $H$ is.
\end{enumerate}
\end{proposition}

Properties of down-up algebras and their invariants are given in
Section \ref{zzsec7}, and in Section \ref{zzsec8} we determine which of the
$A^{Q_i}$ are complete intersections of GK-type for some  
values of the parameters $\alpha$ and $\beta$ (see Table 4 in
Section 8).  In Lemma \ref{zzlem7.3} we note
that the AS Gorenstein and cyclotomic properties of $A^H$ depend on the 
eigenvalues of elements of $H$, and not on the parameters $\alpha$ 
and $\beta$.  The proof of Theorem \ref{zzthm0.3} concludes in Section 
\ref{zzsec9}, where the remaining cases, including $A(\alpha, -1)$ and  
$A(0,1)$ (that have non-diagonal automorphisms), are completed;
the filtered-graded results allow us to consider only finite subgroups of
$SL_2(k)$ and the groups $Q_i$.

Throughout, an element of a graded algebra  usually means a homogeneous 
element. For simplicity, the base field $k$ is assumed to be 
algebraically closed of characteristic zero, though some
definitions and basic properties do not require it.
This paper is closely related to \cite{KKZ4}, in which several
different definitions of ``noncommutative complete intersection''
are proposed.

\section{Preliminaries}
\label{zzsec1}

Most of the notation and
the terminology follows that of \cite{JoZ, KKZ1, KKZ4}. For example, 
$\Aut(A)$ denotes the group of graded
algebra automorphisms of $A$. For $g\in \Aut(A)$, $Tr_A(g,t)$ denotes
the trace function of $g$ (as defined in \cite[p. 6335]{KKZ1}).

We refer to \cite{KL} for the definition of Gelfand-Kirillov
dimension and its basic properties. However, for our purpose, we
will use the following simpler definition.

\begin{definition}
\label{zzdef1.1}
Let $A$ be a connected graded algebra. The {\it Gelfand-Kirillov
dimension} (or {\it GK-dimension} for short) of $A$ is defined
to be
\begin{equation}
\label{E1.1.1}\tag{E1.1.1}
\GKdim A=\limsup_{n\to\infty}
\frac{\log (\sum_{i=0}^n \dim A_i)}{\log n}.
\end{equation}
\end{definition}
This definition agrees with the original definition given in \cite{KL}
when $A$ is finitely generated, but may differ otherwise. Note that
most of our algebras will be noetherian, hence finitely generated.

We recall the definition of AS regularity, see also
\cite[Definition 1.5]{KKZ1}.

\begin{definition}
\label{zzdef1.2}
Let $A$ be a connected graded algebra.
We call $A$ {\it Artin-Schelter Gorenstein} (or {\it AS Gorenstein}
for short) of dimension $d$ if the following conditions hold:
\begin{enumerate}
\item[(a)]
$A$ has injective dimension $d<\infty$ on the left and
on the right,
\item[(b)]
$\Ext^i_A(_Ak,_AA)=\Ext^i_{A}(k_A,A_A)=0$ for all $i\neq d$, and
\item[(c)]
$\Ext^d_A(_Ak,_AA)\cong \Ext^d_{A}(k_A,A_A)\cong k(l)$ for some $l$ (where
$l$ is called the {\it AS index}).
\end{enumerate}
If, in addition,
\begin{enumerate}
\item[(d)]
$A$ has finite global dimension, and
\item[(e)]
$A$ has finite Gelfand-Kirillov dimension,
\end{enumerate}
then $A$ is called {\it Artin-Schelter regular}
(or {\it AS regular} for short) of dimension $d$.
\end{definition}

The homological determinant $\hdet: \Aut(A)\to k^\times$
was introduced in \cite{JoZ}. The following lemma
will be used frequently.

\begin{lemma}
\label{zzlem1.3}\cite[Lemma 2.6]{JoZ}
Let $A$ be AS Gorenstein of injective dimension $d$ and let
$g\in \Aut(A)$. Suppose $g$ is $k$-rational in the sense of
\cite[Definitions 1.3]{JoZ} {\rm{(}}e.g., if $A$ is AS regular, 
or $A$ is PI{\rm{)}}. Then the trace of $g$ is of the form
\begin{equation}
\label{E1.3.1}\tag{E1.3.1}
Tr_A(g,t)=(-1)^d (\hdet g)^{-1} t^l+{\text{lower degree terms}},
\end{equation}
when we express $Tr_A(g,t)$ as a Laurent series in $t^{-1}$.
\end{lemma}

\begin{remark}
\label{zzrem1.4}
In this paper, one can view \eqref{E1.3.1} as a definition of
homological determinant, or as a way of computing the homological
determinant. Hence we will not review the original definition of the
homological determinant in \cite{JoZ}, which uses local cohomology.
Let $G$ be a subgroup of $\Aut(A)$.
We say the homological determinant $\hdet$ of a $G$-action on $A$
is trivial if $\hdet g=1$ for all $g\in G$.

We make the
following general hypotheses for the rest of the paper:
\begin{enumerate}
\item[(H1.4.1)]
{\bf Every finitely generated graded $A$-module is rational over $k$
in the sense of \cite[Definition 1.4]{JoZ}, and}
\item[(H1.4.2)]
{\bf Every $g\in \Aut(A)$ is $k$-rational in the sense of
\cite[Definition 1.3]{JoZ}.}
\end{enumerate}
Similar assumptions were made in \cite[Lemma 4.7]{KKZ2}.
By \cite{JoZ} every finitely generated graded
module over $A$ is rational over $k$ and every $g\in \Aut(A)$ is
$k$-rational if one of the following holds:
\begin{enumerate}
\item
$A$ is noetherian PI;
\item
$A$ is a factor ring of a noetherian AS regular algebra;
\item
$A$ is a fixed subring of a noetherian AS regular algebra under
a finite group action.
\end{enumerate}
 All the algebras
in our main theorems satisfy the rationality conditions given in
\cite[Definitions 1.3 and 1.4]{JoZ}, and for this reason
we can freely use results in \cite{JoZ}.
\end{remark}

The following definition is needed in the next section.

\begin{definition}
\label{zzdef1.5} A connected graded algebra $A$ is called a
{\it hypersurface} if $A\cong C/(\Omega)$ where $C$ is a
noetherian AS regular algebra and $\Omega$ is a regular normal
element in $C$.
\end{definition}

The traces of some special automorphisms of a hypersurface can be
easily computed using the following lemma.

\begin{lemma}
\label{zzlem1.6} Let $C$ be a connected graded algebra with a
regular normal element $\Omega\in C$. Let $g\in \Aut(C)$ such that
$g(\Omega)= \lambda \Omega$ for some $\lambda\in k^\times$, and let
$h$ be the induced automorphism of $g$ on $\overline{C}:=C/(\Omega)$. Then
\begin{enumerate}
\item
$$Tr_C(g,t)= \frac{Tr_{\overline{C}}(h, t)}{1-\lambda t^{\deg \Omega}}.$$
\item
Suppose that there is a sequence of elements $\{x_1, \cdots, x_n\}$
in $C$ such that, for each $i$,
\begin{enumerate}
\item
 the image of $x_i$ in the factor ring
$C/(x_1,\cdots,x_{i-1})$ is a regular normal element,
\item
$C/(x_1,\cdots,x_n)=k$, and
\item
for each $i$, $g(x_i)=\lambda_i x_i$ for some $\lambda_i\in k^\times$.
\end{enumerate}
Then $$Tr_C(g,t)=\frac{1}{(1-\lambda_1 t^{\deg x_1})\cdots
(1-\lambda_n t^{\deg x_n})}.$$
As a consequence, $\hdet g=\prod_{i=1}^n \lambda_i$.
\item
Suppose that $C$ is as in part (2). Then 
$$Tr_{\overline{C}}(h,t)=\frac{1-\lambda
t^{\deg \Omega}}{(1-\lambda_1 t^{\deg x_1})\cdots (1-\lambda_n t^{\deg
x_n})}.$$
As a consequence, $\hdet h=\lambda^{-1}\prod_{i=1}^n \lambda_i$.
\end{enumerate}
\end{lemma}

\begin{proof} (1) Denote the degree $n$ shift of a graded vector space
$M$ by $M[n]$. Applying $g$ to the
short exact sequence
$$0\to C[\deg \Omega]\xrightarrow{r_\Omega} C\to C/(\Omega)\to 0,$$
where $r_{\Omega}$ denotes the left $C$-module homomorphism defined by
the right multiplication by $\Omega$,
we obtain that 
$Tr_C(g,t)-Tr_C(g,t)\lambda t^{\deg \Omega}=Tr_{\overline{C}}(h,t)$.
The assertion follows.

(2) This follows from part (1) and induction on $n$. The second assertion
follows by Lemma \ref{zzlem1.3}.

(3) The first assertion follows from parts (1) and (2), and the
second assertion follows from Lemma \ref{zzlem1.3}.
\end{proof}

A noncommutative version of ``reflection'' was introduced in
\cite{KKZ1}.

\begin{definition}
\label{zzdef1.7} Let $A$ be a noetherian connected graded algebra
with $\GKdim A = n$. An element $g\in \Aut(A)$ is called a {\it
quasi-reflection} if the trace of $g$, $Tr_A(g,t)$, is of the form
$$\frac{p(t)}{(1-t)^{n-1}q(t)}$$
where $p(t), q(t)$ are integral polynomials with $p(1)q(1)\neq 0$.
\end{definition}

We often omit the prefix ``quasi-'' if no confusion occurs. If $A$
is AS regular then $p(t) = 1$.

Let $G\subset \Aut(A)$. The fixed subring of $A$ under the
$G$-action is defined to be
$$A^G:=\{x\in A: g(x)=x,\; {\text{for all $g\in G$}}\}.$$
If $G$ is generated by a single element $g$, we write $A^g$ for $A^G$.

\section{Proof of Theorem \ref{zzthm0.1}}
\label{zzsec2}
In this section a hypersurface means
$C/(\Omega)$, where $C$ is a noetherian AS regular algebra of dimension
three and $\Omega$ is a regular normal element in $C$. In this section 
$G$ is a finite subgroup of $\Aut(A)$.

\begin{proof}[Proof of Theorem \ref{zzthm0.1}]
(2) $\Longrightarrow$ (1)
Let $A^G$ be a hypersurface $C/(\Omega)$.
By Rees's lemma \cite[Proposition 3.4(b)]{Le}, $C/(\Omega)$
is AS Gorenstein.

(1) $\Longrightarrow$ (2)
Conversely, we assume that $A^G$ is AS Gorenstein, where $A$ is
an AS regular algebra of dimension two that is generated in degree
1, and $G$ is a finite subgroup of $\Aut(A)$. We want to show that
$A^G$ is a hypersurface.

Since $k$ is algebraically closed, every AS regular algebra $A$ of
dimension two generated in degree 1 is isomorphic to either
$$k_q[x,y]:=k\langle x,y\rangle/(yx-qxy), \quad {\text{or}}
\quad k_J[x,y]:= k\langle x,y\rangle/(yx-xy-x^2)$$
where $q\in k^\times$.
We analyze fixed subrings $A^G$ in each case. Part of the
analysis in this proof will be used in the proof of Theorem
\ref{zzthm0.3} in later sections.

\medskip

\noindent
{\bf Case 2.1:} $A=k_J[x,y]$. It follows by an easy computation that
every graded algebra automorphism $g$ of $A$ is of the form
$$g: x\to ax, \quad y\to ay+bx$$
for some $a\in k^{\times}$ and $b\in k$. If
$g$ is of finite order, then $g$ is of the form
$$\rho_a: x\to ax, \quad y\to ay$$
for some $a\in k^\times$. The trace of $\rho_a$ is
$$Tr_A(\rho_a,t)=\frac{1}{(1-at)^2}.$$
Hence $\Aut(A)$ does not contain any reflections [Definition \ref{zzdef1.7}].
By \cite[Corollary 4.11]{KKZ2}, for every finite subgroup $G\subset
\Aut(A)$, $A^G$ is AS Gorenstein if and only if the homological
determinant of $G$ is trivial. By Lemma \ref{zzlem1.3}, the
homological determinant of $\rho_a$ is given by
$$\hdet \rho_a= a^2.$$
When the homological determinant is trivial, $a^2=1$. This means that
$a=1$ or $a=-1$. When $G$ is non-trivial, $G=\langle Id_A, \rho_{-1}
\rangle$. The fixed subring $A^G$ is the second Veronese subalgebera
$A^{(2)}$ which is generated by
$$x^2, \quad y^2, \quad xy.$$
By Molien's theorem,
$$H_{A^G}(t)=\frac{1}{2}\biggl(\frac{1}{(1-t)^2}+\frac{1}{(1+t)^2}\biggr)=
\frac{1-t^4}{(1-t^2)^3}.$$
Let $X:=x^2, \; Y:=y^2$ and $Z:=xy$ with $\deg X=\deg Y=\deg Z=2$. It is
straightforward to check that $X, Y$ and $Z \in A^G$ satisfy the 
following relations:
\begin{align}
\label{E2.0.1}\tag{E2.0.1}
Z X&=XZ+2X^2,\\
\label{E2.0.2}\tag{E2.0.2}
Y X&=XY+4XZ+6 X^2,\\
\label{E2.0.3}\tag{E2.0.3}
Y Z&=(Z+2X)Y+2XZ, \quad {\text{and}},\\
\label{E2.0.4}\tag{E2.0.4}
0&= XY+XZ-Z^2.
\end{align}
The first three relations \eqref{E2.0.1}-\eqref{E2.0.3} define an
iterated Ore extension
$$C:=k[X][Z;\delta_1][Y;\sigma_2,\delta_2],$$
for appropriate automorphism $\sigma_2$ and derivations $\delta_1,
\delta_2$, so that $C$ is a noetherian AS regular algebra of dimension 3
with Hilbert series:
$$H_C(t) = \frac{1}{(1-t^2)^3}.$$
Further, by a computation using \eqref{E2.0.1}-\eqref{E2.0.3} one checks that
$\Omega:=XY+XZ-Z^2$ is a regular central element of $C$.
Forming the factor algebra $C/(\Omega)$, we compute Hilbert series:
$$H_{C/(\Omega)}(t) = \frac{1-t^4}{(1-t^2)^3} = H_{A^G}(t).$$
Therefore $A^G\cong C/(\Omega)$, which is a hypersurface. This
completes Case 2.1.

For any noetherian AS Gorenstein algebra $A$ define
$$\AutSL(A)=\{g\in \Aut(A): \hdet g=1\}$$
and
$$\AutSL_{-1}(A)=\{g\in \Aut(A): \hdet g=1\; {\text{or}}\; -1\}.$$
We record the  analysis of $\Aut(k_J[x,y])$ in the following proposition.

\begin{proposition}
\label{zzpro2.1} Let $A=k_J[x,y]$. If $g\in \AutSL(A)$ has finite
order, then $g$  is either $Id_A$ or $\rho_{-1}$.  If $g\in
\AutSL_{-1}(A)$ has finite order, then $g=Id_A$, $\rho_{-1}$,
$\rho_i$, or $\rho_{-i}$ where $i^2=-1$.
\end{proposition}

The following result \cite[Theorem 0.2]{KKZ2}
will be used in the remaining cases.

\begin{lemma}\cite{KKZ2}
\label{zzlem2.2}
Let $A$ be a skew polynomial ring $k_{p_{ij}}[x_1,\cdots,x_n]$, and
let $G$ be a finite subgroup of $\Aut(A)$. Then $A^G$ is AS Gorenstein
if and only if the $\hdet$ of the induced action of $G/R$ on $A^R$
is trivial, where $R$ is the subgroup of $G$ generated by all
reflections in $G$.
\end{lemma}

\medskip
\noindent
{\bf Case 2.2}: $A=k_q[x,y]$ where $q\in k^\times$ and $q\neq \pm 1$.
In this case,  $\Aut(A)=k^\times \times k^\times$, and for any
$g\in \Aut(A)$, there are $a,b\in k^\times$ such that $g(x)=a x$ and
$g(y)=by$. The homological determinant of $g$ is $ab$ [Lemma
\ref{zzlem1.6}(2)]. The following proposition is
now clear.

\begin{proposition}
\label{zzpro2.3} Let $A=k_q[x,y]$ where $q\neq \pm 1$. Then
$$\AutSL(A)=\{g\in \Aut(A)\mid g(x)=a x, g(y)=a^{-1} y,\;
{\text{where}}\; a\in
k^\times\}$$
and
$$\AutSL_{-1}(A)=\{g\in \Aut(A)\mid g(x)=a x, g(y)
=\pm a^{-1} y, \; {\text{where}}\; a\in k^\times\}.$$
\end{proposition}

If $g\in \Aut(A)$ is of finite order, then $g(x)=a_1 x$ and $g(y)=b_1 y$
where $a_1,b_1$ are roots of unity. Every reflection of $k_q[x,y]$ 
for $q \neq \pm 1$ is of the
form $g: x\to x, y\to b_2y$, for some $b_2$, or $g: x\to a_2x, y\to y$,
for some $a_2$. Let $G$ be a finite subgroup of $\Aut(A)$ such that
$A^G$ is AS Gorenstein. Let $R$ be the subgroup of $G$ generated by
reflections in $G$ (note that $R$ is normal in $G$). Then $A^R$ is 
the subring of $A$ generated by
$x^n$ and $y^m$ for some $n,m>0$. The induced $G/R$-action, for any
$\bar{g}\in G/R$, on $B:= A^R$ is of the form
$$\bar{g}: x^n\to a x^n, \quad y^m\to b y^m$$
for some roots of unity $a,b$. The homological determinant of $\bar{g}$
is $ab$ (by Lemma \ref{zzlem1.6}(2)). By Lemma \ref{zzlem2.2}, the
homological determinant of $\bar{g}$ is 1. Hence $b=a^{-1}$. This
implies that $G/R$ is a cyclic group generated by some automorphism
$\phi: x^n\to a x^n, y^m\to a^{-1} y^m$, where $a$ is a primitive $w$th
root of unity for some $w\geq 1$. Thus $(A^R)^{G/R}$ is generated by
$X:=(x^n)^w,\; Y:=x^ny^m$, and $Z:=(y^m)^w$. Then
$$A^G=(A^R)^{G/R}\cong k_{p_{ij}}[X,Y,Z]/(XZ-q^{-nm{w \choose 2}}Y^w),$$
where $p_{12}=q^{nmw}$, $p_{13}=q^{nmw^2}$ and $p_{23}=q^{nmw}$. Note
that $C:=k_{p_{ij}}[X,Y,Z]$ is noetherian and AS regular of dimension
3 and $\Omega:= XZ-q^{-nm{w \choose 2}}Y^w$ is a regular normal
element of $C$. Therefore $A^G\cong C/(\Omega)$, which is a hypersurface.

\medskip
\noindent
{\bf Case 2.3:} $A=k[x,y]$.
 Here $\deg x=\deg y=1$ by assumption. In this
case $\AutSL(A)=SL_2({k})$, and some subgroups of $\AutSL_{-1}(A)$ 
will be discussed in the next section.

We need the following lemma.

\begin{lemma}
\label{zzlem2.4} Let $B=k[X,Y]$ where $\deg X=d_1>0$ and $\deg
Y=d_2>0$. If $G$ is a finite subgroup of $\Aut(B)$ with trivial
homological determinant, then $B^G$ is a {\rm{(}}commutative{\rm{)}} 
hypersurface.
\end{lemma}

\begin{proof} If $d_1=d_2$, the homological determinant agrees with
the usual determinant. By Klein's
classification of finite subgroups $G \subset SL_2({k})$,
\cite{Kl1, Kl2} (see \cite[Chapter 3: Theorem 6.17, p. 404]{Su} for 
general algebraically closed field) and the explicit description of 
the fixed subrings $A^G$, every $A^G$ is a hypersurface. All 
these fixed subrings $A^G$ will
be listed in Section \ref{zzsec3}.

Assume now that $d_1<d_2$. Every $g\in G$  must be of the form
\begin{equation}
\label{E2.4.1}\tag{E2.4.1}
g: X\to a X, \quad Y\to bY+ c X^v,
\end{equation}
where $a,b\in k^\times$, $c\in k$, and
$v=d_2/d_1$ (if it is an integer), and if $d_2/d_1$ is not an 
integer, then $c=0$. By Lemma
\ref{zzlem1.6}(2), $\hdet g=ab$, and since $\hdet$ is trivial by hypothesis,
$a=b^{-1}$. Hence $g: X\to b^{-1} X,\quad Y\to bY +cX^v$. If
$b=1$, then $c$ must be zero as $g$ is of finite order. This implies
that $c$ is uniquely determined by $b$. As a consequence, if $g,h\in
G$, then $ghg^{-1}h^{-1}$ is the identity, and hence $G$ is abelian, and
is determined by the restriction $G\mid_{kX}$. Since every
finite subgroup of $k^\times$ is cyclic, $G$ is generated by
a single automorphism $\phi\neq 1$. Write $\phi$ in the form of
\eqref{E2.4.1} (with $a=b^{-1}$).

If $b^{v+1}\neq 1$, $\phi(Y+c(b-b^{-v})^{-1}X^v)=
b(Y+c(b-b^{-v})^{-1}X^v)$. Replacing $Y$ by $Y+c(b-b^{-v})^{-1}X^v$,
we may assume that $c=0$. Next we consider the case when  $b^{v+1}=1$.
By induction, $\phi^s(Y)=b^s Y+s b^{s-1} cX^v$ for all $s\geq 1$.
So $\phi^{v+1}$ maps $X$ to $X$ and $Y$ to $Y+c(v+1)b^{-1}X^{v}$.
Since $\phi^{v+1}$ has finite order, $c=0$.
In both cases we may assume $c=0$ (after
choosing a new generator $Y$ if necessary). When $c=0$, $B^G$ is
isomorphic to $k[X^w, Y^w, XY]\cong k[X',Y',Z']/(X'Y'-(Z')^w)$,
where $w$ is the order of $b$. Hence $B^G$ is a hypersurface.
\end{proof}

To complete {\bf Case 2.3}, let $G$ be a finite subgroup of $\Aut(A)$
such that $A^G$ is Gorenstein. Then by
the classical Shephard-Todd-Chevalley Theorem $B:=A^R$ is a 
commutative polynomial
ring in two variables, $k[X,Y]$, with $\deg X$ and $\deg Y$ positive. 
By Lemma \ref{zzlem2.2},
the action of $G/R$ on $B$ has trivial homological determinant. By Lemma
\ref{zzlem2.4}, $A^G=B^{G/R}$ is a hypersurface.

\medskip
\noindent
{\bf Case 2.4:} $A=k_{-1}[x,y]$. By \cite[Theorem 5.5]{KKZ3} (also see
Lemma \ref{zzlem2.2}), $A^R$ is isomorphic to a skew polynomial ring
$k_{p}[X,Y]$ with
$\deg X>0$, $\deg Y>0$ and $p = \pm 1$. Since $A$ has PI-degree 2, 
the PI-degree
of $k_p[X,Y]$ is at most two. If the PI-degree of $A^R$ is 1, then $p=1$ and
$A^R$ is commutative, and by Lemma \ref{zzlem2.4}, $A^G=(A^R)^{G/R}$ is
a hypersurface. It remains to consider the case when the PI-degree is
$2$; then $p=-1$. The next lemma will be used to complete the proof 
in this case.

\begin{lemma}
\label{zzlem2.5} Let $B=k_{-1}[X,Y]$ where $\deg X>0$ and
$\deg Y>0$. If $G$ is a finite subgroup of $\Aut(B)$ with trivial
homological determinant, then $B^G$ is a hypersurface.
\end{lemma}

\begin{proof}
If $\deg X=\deg Y$, then, with respect to the basis $\{X,Y\}$,
$$\Aut(A)=\biggl\{\begin{pmatrix} a& 0\\ 0& b\end{pmatrix},
\begin{pmatrix} 0& c\\ d& 0\end{pmatrix}: a,b,c,d\in k^\times\biggr\}$$
 by \cite[Lemma 1.12]{KKZ5}.
It follows from the classification in \cite{CKWZ} that, if the homological
determinant of $G$-action is trivial, then $G$ is either
\begin{enumerate}
\item[(i)]
$C_n=\biggl\langle \begin{pmatrix} a& 0
\\0&a^{-1}\end{pmatrix}\biggr\rangle,$ where $a$ is a primitive $n$th root of
unity, or
\item[(ii)]
$D_{2n}=\biggl\langle \begin{pmatrix} a& 0
\\0&a^{-1}\end{pmatrix}, \begin{pmatrix} 0& 1
\\1&0\end{pmatrix}\biggr\rangle,$ where $a$ is a primitive $n$th root of unity.
\end{enumerate}\

Case (i): If $G=C_n$, then $B^G$ is generated by $X_1:=X^n, X_2:=Y^n$ and
$X_3:=XY$. Similar to the commutative case, $B^G$ is isomorphic to
$$k_{p_{ij}}[X_1,X_2,X_3]/ (X_1X_2-(-1)^{{n \choose 2}} X_3^n),$$
where $p_{ij}$ are 1 when $n$ is even, and $p_{ij}$ are $-1$
when $n$ is odd. So $B^G$ is a hypersurface.

Case (ii):
If $G=D_{2n}$, where $n$ is a positive even integer, then $B^G$ is
generated by
$$X_1:=X^n+Y^n, \; X_2:=(X^n-Y^n)XY, \quad {\text{and}}\quad
X_3:= (XY)^2,$$ and is isomorphic to $k[X_1, X_2,X_3]/(\Omega)$,
where $\Omega:=X_2^2-X_1^2 X_3+4(-1)^{{n \choose
2}}X_3^{\frac{n}{2}+1}$ is a regular element of $k[X_1, X_2,X_3]$.
Therefore $B^G$ is a commutative hypersurface.

If $G=D_{2n}$, where $n$ is a positive odd integer, then
$B^G$ is generated by
$$X_1:=X^n+Y^n, \; X_2:=(X^n-Y^n)XY, \quad {\text{and}}\quad
X_3:= (XY)^2,$$ and $B^G$ is isomorphic to $C/(\Omega)$, where
$$C=k\langle X_1, X_2,X_3\rangle/
([X_3,X_1]=[X_3,X_2]=0,X_2X_1+X_1X_2=-4(-1)^{{n\choose 2}}
X_3^{\frac{n+1}{2}})$$
and $\Omega:=X_2^2+X_1^2 X_3$.
Therefore $C$ is of the form
$$k[X_1,X_3][X_2; \sigma, \delta],$$
and so $C$
is noetherian, AS regular of dimension 3,
and $\Omega$ is a regular normal element of $C$. Hence
$B^G$ is a hypersurface.

Assume now that $\deg X=d_1<d_2=\deg Y$. Every $g\in G$ is of the
form \eqref{E2.4.1} and $v$ is an odd integer if $c\neq 0$
(otherwise $g(X)g(Y)\neq -g(Y)g(X)$). Using the argument in
the proof of Lemma \ref{zzlem2.4}, one can assume that $c=0$.
Then $B^G$ is isomorphic to $k_{p_{ij}}[X_1,X_2,X_3]/
(X_1X_2-(-1)^{{n \choose 2}}X_3^n)$, where $p_{ij}$ is either 1
or $-1$. (See the argument in case (i).) Therefore $B^G$ is a
hypersurface, completing the proof of Lemma \ref{zzlem2.5}.
\end{proof}

By combining all these cases, $A^G=B^{G/R}$ is a hypersurface, where
$B=A^R$. This finishes the proof of ``(1) $\Longrightarrow$ (2)'', and
therefore Theorem \ref{zzthm0.1} follows.
\end{proof}

\begin{remark}
\label{zzrem2.6} A special case of Lemma \ref{zzlem2.5}
is when $G=D_{2n}$ and $n=1$. Then
$G$ is generated by the matrix $\begin{pmatrix} 0& 1
\\1&0\end{pmatrix}$. The fixed subring $B^G$ is generated by
$X_1:=X+Y$ and $X_2:= (X-Y)XY$, as $X_3:=(XY)^2= -\frac{1}{4}(X_1X_2+X_2X_1)$
is generated by $X_1$ and $X_2$. In this case, $B^G\cong C/(\Omega)$
where $C$ is the down-up algebra $A(0,1)$
$$C:=k\langle X_1, X_2\rangle/(X_1^2X_2-X_2X_1^2,X_2^2X_1-X_1X_2^2),$$
and where $\Omega:=X_2^2-\frac{1}{4}X_1^2(X_1X_2+X_2X_1)$ is a regular
central element in $C$.

When $G=D_{2}$ and $\deg X=\deg Y=1$, the fixed subring $B^G$ is generated
by two elements of degrees $1$ and $3$, respectively, and hence the classical
Noether bound \cite{No}, $|G| = 2$, on the degrees of generators of $A^G$, 
fails.
\end{remark}

\section{Involutions on ADE singularities}
\label{zzsec3}
The action of $GL_2({k})$ on the vector space
$kt_1+kt_2$ extends naturally to an action of $GL_2(k)$ on the
commutative polynomial ring $k[t_1,t_2]$. The main goal of this
section is to find all finite subgroups
$H\subset GL_2({k})$ that satisfy the following
exact sequence
\begin{equation}
\label{E3.0.1}\tag{E3.0.1}
1\to H\cap SL_2({k})\to H\xrightarrow{\det} \{\pm 1\}\to 1,
\end{equation}
and to characterize the action of $g\in H\setminus G$,
where $G=H\cap SL_2({k})$, on the fixed subring
$k[t_1,t_2]^G$. This analysis is needed for the proof of the main
result Theorem \ref{zzthm0.3}.

Throughout this paper we will use the following matrices, where 
$\epsilon$ is a root of unity, so we record them for
future reference in the table below.


$$\begin{array}{|c|c|c|c|}
\hline
 & & & \\
{\mathbb I}:=\begin{pmatrix} 1& 0\\0&1\end{pmatrix} & 
s:= \begin{pmatrix} 0& 1\\1&0\end{pmatrix} & 
s_{1}:= \begin{pmatrix} 0& 1\\-1&0\end{pmatrix} & 
s_{2}:= \begin{pmatrix} 0& -1\\1&0\end{pmatrix}\\
 & & &\\
\hline
 & & & \\
 d_1:=\begin{pmatrix} -1& 0\\0&1\end{pmatrix} &
 d_2:= \begin{pmatrix} 1& 0\\0&-1\end{pmatrix} &
 c_{\epsilon}:= \begin{pmatrix} \epsilon& 0\\0& \epsilon^{-1}\end{pmatrix} &
c_{\epsilon,-}:= \begin{pmatrix} -\epsilon& 0\\0& \epsilon^{-1}\end{pmatrix}\\
  & & & \\
\hline
\end{array}$$

\begin{center}
Table 1: Elements of $GL_2(k)$
\end{center}

\medskip

Next we recall Klein's classification of the finite subgroups of 
$SL_2({k})$ (up to conjugation) \cite{Kl1, Kl2,Su}, and the generators 
that will be used in our work.  The generators for
the type $E$ groups will not be needed in this paper.

\begin{enumerate}
\item[($A_n$)]
The cyclic groups of order $n+1$, $C_{n+1}$, generated by $c_{\epsilon}$,
where $\epsilon$ is a primitive
$(n+1)$st root of unity, for $n\geq 1$.
\item[($D_n$)]
The binary dihedral groups $BD_{4(n-2)}$ of order $4(n-2)$ generated by
$s_1$ and $c_{\epsilon}$,
where $\epsilon$ is a primitive $2(n-2)$nd root of unity, for $n\geq 4$.
\item[($E_6$)]
The binary tetrahedral group $BT_{24}$ of order $24$.
\item[($E_7$)]
The binary octahedral group $BO_{48}$ of order $48$.
\item[($E_8$)]
The binary icosahedral group $BI_{120}$ of order $120$.
\end{enumerate}

The fixed subrings $k[t_1,t_2]^G$, where $G$ is a finite subgroup of
$SL_2({k})$, are described as ADE-type singularities, and are also
known as du Val singularities, Kleinian singularities, simple
surface singularities, or rational double points. For each finite
subgroup $G\subset SL_2({k})$, the fixed subring
$S:=k[t_1, t_2]^G$ is a hypersurface of the form
$k[x,y,z]/(f)$, namely, a commutative polynomial algebra generated by variables
$x,y,z$ modulo one relation $f$. We give the relation $f$ for each
type of group \cite[Theorem 6.18]{LW}:
\begin{enumerate}
\item[($A_n$)]
$f= x^2+y^2+z^{n+1}$ for $n\geq 1$ with $\deg x=n+1, \deg y=n+1$
and $\deg z=2$;
\item[($D_n$)]
$f=x^2+y^2z+z^{n-1}$ for $n\geq 4$ with $\deg x=2(n-1), \deg y=2(n-2)$
and $\deg z=4$.
\item[($E_6$)]
$f=x^2+y^3+z^4$ with $\deg x=12$, $\deg y=8$ and $\deg z=6$.
\item[($E_7$)]
$f=x^2+y^3+yz^3$ with $\deg x=18$, $\deg y=12$ and $\deg z=8$.
\item[($E_8$)]
$f=x^2+y^3+z^5$ with $\deg x=30$, $\deg y=20$ and $\deg z=12$.
\end{enumerate}

The first step in determining the action of $g\in H\setminus G$
is to compute its homological determinant on $k[t_1,t_2]^G$.

\begin{lemma}
\label{zzlem3.1}
Suppose $H$ is a finite subgroup that satisfies \eqref{E3.0.1}.
Let $g\in H\setminus G$, where $G=H\cap SL_2({k})$.
Then the homological determinant of $g$ on $k[t_1,t_2]^G$ is $-1$.
\end{lemma}

\begin{proof} Let $A=k[t_1,t_2]$. Since $G\subset SL_2({k})$,
$S:=A^G$ is AS Gorenstein \cite[Theorem 3.3]{JoZ} and the AS index of
$S$ is equal to the AS index of $A$ \cite[Lemma 2.6(b)]{CKWZ}.

It follows from \eqref{E3.0.1} that the induced action of $g$ on $S$
is an involution. Since the homological determinant is a homomorphism, 
$\hdet|_{S} g$ is either 1 or $-1$. We will
show that it is not $1$ by contradiction, and so we assume that 
$\hdet|_{S} g =1$.
Then $S^g$ is AS Gorenstein \cite[Theorem 3.3]{JoZ} and the AS index of
$S^g$ is equal to the AS index of $S$ by \cite[Lemma 2.6(b)]{CKWZ}.
Thus the AS index of $A^H(=S^g)$ is equal to the AS index of $A$.
It follows from \cite[Lemma 2.6(b)]{CKWZ} that $\det \phi=1$ for all
$\phi\in H$, a contradiction. Therefore $\hdet\mid_S g=-1$.
\end{proof}

In our analysis of the invariants $A^H :=k[t_1,t_2]^H$, we will 
consider $A^H = (A^G)^h:=S^h$, where $G := H \cap SL_2(k)$, $S:=A^G$, 
and $h$ is the automorphism of $S$ induced by $g \in H \setminus G$.  
Hence it suffices to consider $h$ acting on the algebras
$S:=k[x,y,z]/(f)$ listed as above.  By Lemma \ref{zzlem3.1} we need 
to describe only involutions of $S$ with negative homological 
determinant.  We now consider each of these cases;
for most of the cases we will not give an explicit description of the group
$H$, as it will not be needed in later computation.

\medskip
\noindent
{\bf Case $A_{n}$:}  For the case where $H \cap SL_2(k)$ is of type 
$A_n$, in Lemma \ref{zzlem3.2} we compute
the groups $H$ and determine that there are five cases.  Then in 
Lemma \ref{zzlem3.3} we compute, for each case, the fixed subring 
under $G$ (which is a hypersurface) and give a description of the 
automorphism $h$.


\begin{lemma}
\label{zzlem3.2} If $H$ is a finite subgroup of $GL_2(k)$
satisfying \eqref{E3.0.1} such that $G: =H\cap SL_2(k)$
is cyclic, then, up to conjugation, $H$ is one of the following groups.
\begin{enumerate}
\item[($A_{n,1}$)]
$H=\{ {\mathbb I}, d_{1}\}=C_2$.
\item[($A_{n,2}$)]
$H$ is generated by $d_1$ and $c_{\epsilon}$, where
$\epsilon$ is a primitive $2n$th root of unity,
and $H=C_{2n}\times C_2$.
\item[($A_{n,3}$)]
$H$ is generated by $d_1$ and $c_{\epsilon}$, where
$\epsilon$ is a primitive $n$th root of unity and $n$
is an odd integer. In this case $H=C_{n}\times C_2\cong
C_{2n}$.
\item[($A_{n,4}$)]
$H$ is generated by $c_{\epsilon, -}$,
where $\epsilon$ is a primitive $4n$th root of unity, and
$H=C_{4n}$. Note that $H$ does not contain any (nontrivial)
reflection of the vector space $kt_1+kt_2$.
\item[($A_{n,5}$)]
$H$ is generated by $s$ and $c_{\epsilon}$, where $\epsilon$ is a 
primitive $n$th root of unity, and $H$ is the dihedral group 
$D_{2n}$ of order $2n$.
\end{enumerate}
Note that, except for case {\rm{($A_{n,4}$)}}, $H$ contains either
$d_1$ or $s$, which are classical reflections of $kt_1+kt_2$.
\end{lemma}

\begin{proof} Let $g$ be any element in $H \setminus G$.

If $G$ is trivial, then $g^2={\mathbb I}$. Since $\det g=-1$
[Lemma \ref{zzlem3.1}], and
$g=d_1$ up to conjugation. This is case ($A_{n,1}$).

If $G$ has order $2$, then $G=\{{\mathbb I},-{\mathbb I}\}$. In this
case $g$ has order 2 or $4$. If $g$ has order 2, then $H$ is as in case
($A_{n,2}$) for $n=1$. If $g$ has order 4, then $g^2=-{\mathbb I}$,
so $g=c_{\epsilon,-}$, where $\epsilon$ is either $i$ or $-i$.
This is case ($A_{n,4}$) for $n=1$.

Now assume that $G$ has order $n>2$. Up to conjugation,
$G$ is generated by $c_{\epsilon}$, for $\epsilon$ a primitive $n$th
root of unity. Write $g$ as $\begin{pmatrix}a&b\\c&d\end{pmatrix}$.
Since $gc_{\epsilon}g^{-1}=c_{\epsilon}^w$ for some $w$, a matrix
computation shows that either $a=d=0$ or $b=c=0$.  If $b=c=0$,
$g=\begin{pmatrix}-a&0\\0&a^{-1}
\end{pmatrix}$ for some root of unity $a$. Since $g^2\in G$,
the order of $a$ divides $2n$. If the order of $a$ divides $n$, then
$\begin{pmatrix}a&0\\0&a^{-1}\end{pmatrix}$ is in $G$. So we can
choose $g=d_1$, which is case ($A_{n,2}$) or ($A_{n,3}$).
Suppose next that $a$ has order $2k$,
where $k$ divides $n$ but $2k$ does not divide $n$, so we can write 
$2k = 2^{m+1}k'$ with $k'$ not divisible by $2$, and
$n=2^m\ell$ for $\ell$ not divisible by $2$.  Then
the order of $g^{k'}c_\epsilon^{2^m}$ is $2n$, so we can assume 
$a$ has order $2n$. Finally, if $a$ has
order $2n$, then $H$ is generated by $g$. This is either case ($A_{n,4}$)
or case ($A_{n,3}$), depending upon the parity of $n$.

If $a=d=0$, up to conjugation we have $g=s$. This is case ($A_{n,5}$).
\end{proof}

Let $S$ be the fixed subring $k[t_1, t_2]^{G}$, where $G$ is as
in Lemma \ref{zzlem3.2}.
Then (except in case $(A_{n,1})$ where $G = \{\mathbb{I}\}$),  we 
have $S$ is isomorphic to $k[X_1, X_2, X_3]/(X_1X_2-X_3^{m})$,
where $m$ is the order of the cyclic group $G$ (in Lemma
\ref{zzlem3.2} cases $(A_{n,3})$ and $(A_{n,5})$ we have  $G=C_{n}$, while
in cases $(A_{n,2})$ and $(A_{n,4})$ we have  $G=C_{2n}$). Let $h$ be 
the graded algebra automorphism of $S$ induced by
$g\in H\setminus G$. We next describe $h$ in each of the five cases 
$(A_{n,1})$-$(A_{n,5})$  of Lemma \ref{zzlem3.2},
and we summarize the results in Table 2 below.

\medskip

\begin{enumerate}
\item[($A_{n,1}$):]
Here $G =\{ \mathbb{I} \}$, so $S=k[t_1, t_2]$ and $h: t_1\to -t_1, 
t_2\to t_2$.
The fixed subring $S^h$ is $k[t_1^2, t_2]$.
\end{enumerate}

\medskip

To obtain the hypersurface $f$ given in the classification of Kleinian 
singularities above Lemma \ref{zzlem3.1}, in some of the cases we will 
change variables, and set $$x:=\frac{1}{\sqrt{2}}(X_1+X_2),\;\;
y:=\frac{1}{\sqrt{2}} (X_1-X_2) \text{ and } z:=\xi X_3,$$ where $\xi^n=-1$ (or
$\xi^{2n}=-1$).

\medskip

\begin{enumerate}
\item[($A_{n,2}$):]
Let $X_1:=t_1^{2n} , X_2:=t_2^{2n}$ and $X_3:=t_1t_2$ then\\
$S=k[X_1, X_2, X_3]/(X_1X_2-X_3^{2n})$ and
$$h: X_1\to X_1, \quad X_2\to X_2, \quad X_3\to -X_3.$$
Using the new variables $x,y,z$, we have
$S=k[x,y,z]/(x^2+y^2+z^{2n})$ and
$$h: x\to x, \quad y\to y, \quad z\to -z.$$
In this case we have the same relation $f$ given for the Kleinian singularity
for $C_{2n}$.  Note that $h$ lifts to a reflection of $k[x,y,z]$.
The fixed subring $S^h$ is a subring generated by $x,y$, and $z^2$,
which is isomorphic to a hypersurface $k[x,y,z']/(x^2+y^2+(z')^{n})$.
\item[($A_{n,3}$):]
Let $X_1:=t_1^n , X_2:=t_2^{n}$ and $X_3:=t_1t_2$,
then $S=k[X_1, X_2, X_3]/(X_1X_2-X_3^{n})$, where $n$ is odd, and
$$h: X_1\to -X_1, \quad X_2\to X_2, \quad X_3\to -X_3.$$
In this case, $h$ does not lift to a reflection of $k[X_1,X_2,X_3]$,
and $h$ does not preserve the relation $f$ given in the description of the
Kleinian singularities above. 
\item[($A_{n,4}$):]
Let $X_1:=t_1^{2n} , X_2:=t_2^{2n}$ and $X_3:=t_1t_2$, then\\
$S=k[X_1, X_2, X_3]/(X_1X_2-X_3^{2n})$ and
$$h: X_1\to -X_1, \quad X_2\to -X_2, \quad X_3\to -X_3.$$
Using the new variables,
$S=k[x,y,z]/(x^2+y^2+z^{2n})$ and
$$h: x\to -x, \quad y\to -y, \quad z\to -z.$$
In this case, we have the same relation $f$ as in the
description of the Kleinian singularities, but $h$ does not lift to a 
reflection of $k[x,y,z]$.
\item[($A_{n,5}$):]
Let $X_1:=t_1^n , X_2:=t_2^{n}$ and $X_3:=t_1t_2$ and
$S=k[X_1, X_2, X_3]/(X_1X_2-X_3^n)$ and
$$h: X_1\to X_2, \quad X_2\to X_1, \quad X_3\to X_3.$$
Using the new variables,
$S=k[x,y,z]/(x^2+y^2+z^n)$ and
$$h: x\to x, \quad y\to -y, \quad z\to z.$$
In this case, $h$ lifts to a reflection of $k[x,y,z]$, and we have the same
relation $f$ as in the description of the Kleinian singularities.
\end{enumerate}

This completes the analysis in Case $A_n$, and we summarize it in 
Table 2 below.

$$\begin{array}{|c|c|c|c|}
\hline

\text{ Case } & H & S= A^G &  h \\
\hline
  A_{n,1}& \langle d_1 \rangle & k[t_1,t_2]&  t_1 \to t_1, t_2 \to -t_2 \\
  \hline
 A_{n,2}& \langle d_1, c_{\epsilon_{2n}} \rangle& 
k[x,y,z]/(x^2 +y^2 +z^{2n})& x \to x, y \to y, z \to -z \\
 \hline
  & & &  X_1 \rightarrow -X_1, X_2\rightarrow X_2,\\
   A_{n,3}& \langle d_1, c_{\epsilon_{n}} \rangle, n \text{ odd} & 
k[X_1,X_2,X_3]/(X_1X_2-X_3^{n})& X_3\rightarrow -X_3  \\
  \hline
   A_{n,4}& \langle d_1, c_{\epsilon_{4n},-} \rangle &
k[x,y,z]/(x^2+ y^2+ z^{2n}) &   
x \rightarrow -x, y\rightarrow -y, z\rightarrow -z  \\
   \hline
    A_{n,5}&\langle s, c_{\epsilon_n}\rangle  & 
k[x,y,z]/(x^2+ y^2+ z^{n}) &  
x \rightarrow x, y\rightarrow -y, z\rightarrow z \\
\hline
\end{array}
$$
\begin{center}
Table 2: $H \cap SL_2(k)$ is of type $A_n$ (cyclic)
\end{center}

\medskip

For types $D$ and $E$, we describe
$h\in \Aut(S)$ without calculating $H$, since that is all the information
about $H$ that is needed for Theorem \ref{zzthm0.3}.
We continue to assume $H$ satisfies \eqref{E3.0.1} and
set $G=H\cap SL_2({k})$; let $g\in H\setminus G$ and $h\in \Aut(S)$ 
be the automorphism of
the fixed subring $S=k[t_1,t_2]^G$ that is induced by $g$. Note that $h$ is
an involution of $S$.

\medskip
\noindent
{\bf Case $D_{n}$:}
Suppose $G=BD_{4n}$ for $n \geq 2$.  Assume $h\in \Aut(S)$.
We first consider the case of an even integer $n\geq 4$, and we 
write $n=2m$ for
some $m \geq 2$.  Then since $\deg z<\min\{\deg x, \deg y\}$,
$h(z)=az$ for some $a\in k^\times$. Since the degree of $x$ is
$2(2m-1)$, which is not a multiple of $\deg z$ and $\deg y$, we
have $h(x)=bx$ for some $b\in k^\times$. Finally
$$h(y)=cy+ d z^{m-1}$$
for some $c\in k^\times $ and $d\in k$.
Assume that $h$ is an involution. Then $a^2=b^2=c^2=1$.  Since
$h$ maps the relation to a scalar multiple of
the relation, we have $a=1$ and $d=0$.
If $n$ is an odd integer $n \geq 3$, a similar argument shows that 
any involution $h$ of $S$ is of the form
$$h: x\to b x, \quad y\to c y,\quad z\to z,$$
where $b^2=c^2=1$. Therefore we have three subcases to consider:

\begin{enumerate}
\item[($D_{n,0}$):]
$c=b=-1$. In this case, the homological determinant of $h$ is 1, so this
case does not need to be considered.
\item[($D_{n,1}$):]
$c=-1$ and $b=1$. In this case, the homological determinant of
$h$ is $-1$. Further, $h$ lifts to a reflection of $k[x,y,z]$ and
preserves the relation $f$.
\item[($D_{n,2}$):]
$c=1$ and $b=-1$. In this case, the homological determinant of
$h$ is $-1$.  Further, $h$ lifts to a reflection of $k[x,y,z]$ and
preserves the relation $f$.
\end{enumerate}

It remains to consider the case when $G=BD_8$, the quaternion group of 
order 8. In this case since deg$(y)$ = deg$(z)$, we have $h(x) = bx, 
h(y) = cy + dz$ and  $h(z) = ey + az$.
If $h$ is an involution that maps $f$ to a scalar multiple of $f$ and 
has homological determinant $-1$, there are the following
additional cases to consider:

\begin{enumerate}
\item[($D_{2,3}$):] In this case
$$h: x \to x, \quad y \to y/2-3zi/2, \quad  z \to yi/2-z/2.$$
 Hence the homological
determinant of $h$ is $-1$, $h$ lifts to a reflection of $k[x,y,z]$, 
and $h$ preserves the relation $f$.

\item[($D_{2,4}$):] In this case
$$h: x \to x, \quad  y \to y/2+3zi/2, \quad  z \to-yi/2-z/2.$$
 Hence the homological
determinant of $h$ is $-1$, $h$ lifts to a reflection of $k[x,y,z]$, 
and $h$ preserves the relation $f$.
\end{enumerate}

In cases $D_{2,3}$ and $D_{2,4}$ we show that there is no corresponding 
group $G$ that satisfies \eqref{E3.0.1} since one can take 
$x= t_1^5t_2 - t_1t_2^5$, $y=\alpha (t_1^4+t_2^4)$ and 
$z= \beta t_1^2t_2^2$ for appropriate
constants $\alpha$ and $\beta$ (so that $x^2+y^2z+z^3 = 0$) 
and show that there is no automorphism $g$ of $k[t_1,t_2]$ that induces
the $h$ of cases $D_{2,3}$ or $D_{2,4}$ on $S=k[x,y,z]/(x^2+y^2z+z^3)$.  
One first shows that if such a $g$ exists
it must be of the form $\begin{pmatrix} a& 0\\0& b\end{pmatrix}$ or 
$\begin{pmatrix} 0& c\\d& 0\end{pmatrix},$
and then one shows that no such $g$ induces $h$.

\medskip
\noindent
{\bf Case $E_6$:} Similar to the case of $D_n$, every nontrivial graded
algebra automorphism $h$ of $S$ is of the form
$$h: x\to cx+dy^2, \quad y\to by, \quad z\to az$$
for some $a,b,c\in k^\times$. If $h$ is an involution, then
$a^2=b^2=c^2=1$ and $d=0$. The relation of $E_6$ implies that $b=1$. Therefore
we have possibly three cases to consider:

\begin{enumerate}
\item[($E_{6,0}$):]
$c=a=-1$. In this case, the homological determinant of $h$ is 1.
\item[($E_{6,1}$):]
$c=-1$ and $a=1$. In this case, the homological determinant of $h$ is
$-1$. Further, $h$ lifts to a reflection of $k[x,y,z]$ and preserves
the relation $f$.
\item[($E_{6,2}$):]
$c=1$ and $a=-1$. In this case, the homological determinant of $h$ is
$-1$. Further, $h$ lifts to a reflection of $k[x,y,z]$ and preserves
the relation $f$.
\end{enumerate}

\medskip
\noindent
{\bf Case $E_7$:} Similar to the case of $E_6$, every nontrivial graded algebra
automorphism $h$ of $S$ is of the form
$$h: x\to cx,\quad  y\to by, \quad z\to az$$
for some $a,b,c\in k^\times$. If $h$ is an involution, then
$a^2=b^2=c^2=1$. The relation of $E_7$ implies that
$a=b=1$ and that $c=-1$. The homological determinant of $h$ is
$-1$. Further, $h$ lifts to a reflection of $k[x,y,z]$ and
preserves the relation $f$.

\medskip
\noindent
{\bf Case $E_8$:} Let $h\in \Aut(S)$. The same argument as above
shows that
$$h: x\to cx,\quad  y\to by, \quad z\to az$$
for some $a,b,c\in k^\times$. If $h$ is an involution, then
$a^2=b^2=c^2=1$. The relation of $E_8$ implies that
$a=b=1$ and that $c=-1$. In this case, the homological determinant
of $h$ is $-1$. Further, $h$ lifts to a reflection of $k[x,y,z]$ and
preserves the relation $f$.

The above analysis gives the following result.

\begin{lemma}
\label{zzlem3.3} Suppose $H$ is a finite subgroup of $GL_2({k})$ 
that satisfies \eqref{E3.0.1}, and let $G=H\cap
SL_2({k})$. Write $S:=k[t_1,t_2]^G=k[X_1,X_2,X_3]/(f)$.
Let $g\in H\setminus G$ and let $h$ be the induced involution
of $g$ on $S$. Then $h$ lifts to a reflection of $k[X_1,X_2,X_3]$
and preserves $f$ in the following cases:
\begin{enumerate}
\item
$G$ is of type $(A_{n,1})$, $(A_{n,2})$, or $(A_{n,5})$.
\item
$G$ is of type $D_n$ for all $n\geq 4$, $E_6$, $E_7$, or $E_8$.
\end{enumerate}
\end{lemma}

\begin{proof} By Lemma \ref{zzlem3.1}, the homological determinant
of $h$ on $S$ is $-1$. So $(D_{n,0})$ and  $(E_{n,0})$ cannot
happen. The other cases were checked in the analysis above.
\end{proof}

Let $O$ be the subgroup of $GL_2(k)$ consisting of elements of the
form $\begin{pmatrix} a& 0\\0& b\end{pmatrix}$; then $O\cong
k^\times \times k^\times$. Let $U$ be the subgroup of $GL_2(k)$
consisting of elements of the form $\begin{pmatrix} a& 0\\0&
b\end{pmatrix}$ or $\begin{pmatrix} 0& c\\d&
0\end{pmatrix}$. There is a surjective group homomorphism from $U$
to $\{ {\mathbb I}, s\}$ defined by sending $\begin{pmatrix} a& 0\\0&
b\end{pmatrix}$ to ${\mathbb I}$ and $\begin{pmatrix} 0& c\\d&
0\end{pmatrix}$ to $s$. The kernel of this map is $O$. Hence there
is a short exact sequence of groups
\begin{equation}
\label{E3.3.1}\tag{E3.3.1}
1\to O\to U\to \{ {\mathbb I}, s\}\to 1.
\end{equation}
It is well-known that $\Aut(k_q[t_1,t_2])=O$ if $q\neq \pm 1$, and
$\Aut(k_{-1}[t_1,t_2])=U$ by \cite[Lemma 1.12]{KKZ5}. The following
lemma follows from the proof of Lemma \ref{zzlem3.2}, so its proof
is omitted.

\begin{lemma}
\label{zzlem3.4} Let $H$ be a finite subgroup of $O$ such that
either $H$ is in $SL_2({k})$ or satisfies the short exact
sequence \eqref{E3.0.1}. Then, up to conjugation of $kt_1+kt_2$,
$H$ is equal to one of the following groups.
\begin{enumerate}
\item
The group $C_n$ generated by $c_{\epsilon}$, where $\epsilon$ is a
primitive $n$th root of unity. This group is denoted by $Q_1$.
\item
The group $C_{2n}\times C_2$ generated by $d_1$ and $c_{\epsilon}$,
where $\epsilon$ is a primitive $2n$th root of unity. This group is
denoted by $Q_2$.
\item
The group $C_{n}\times C_2\cong C_{2n}$ 
generated by $d_1$ and
$c_{\epsilon}$, where $\epsilon$ is a primitive $n$th root of unity
and $n$ is an odd integer. This group is denoted by $Q_3$.
\item
The group $C_{4n}$ generated by $c_{\epsilon, -}$, where $\epsilon$
is a primitive $4n$th root of unity. This group is denoted by $Q_4$.
\end{enumerate}
\end{lemma}

Note that $Q_1$ occurred in case ($A_n$), $Q_2$ in case
($A_{n,2}$), $Q_3$ in case ($A_{n,3}$), and $Q_4$ in case
($A_{n,4}$).

\begin{lemma}
\label{zzlem3.5} Let $H$ be a finite subgroup of $U$ that is not
a subgroup of $O$. Suppose that either $H$ is in $SL_2({k})$ or
satisfies the short exact sequence \eqref{E3.0.1}. Then, up to
conjugation, $H$ is one of the following groups.
\begin{enumerate}
\item[(1)]
The binary dihedral group $BD_{4n}$, for $n\geq 1$, that is generated
by $s_1$ and $c_{\epsilon}$, where $\epsilon$ is a primitive $2n$th
root of unity. This group is denoted by $Q_5$. Note that 
$BD_4 = \langle s_1 \rangle$ is
cyclic.
\item[(2)]
The dihedral group $D_{2n}$, for $n\geq 1$, that is generated by $s$ and
$c_{\epsilon}$, where $\epsilon$ is a primitive $n$th root of unity.
This group is denoted by $Q_6$.
\item[(3)]
The group of order $8n$ generated by $d_1$, $s$ {\rm{(}}or $s_1$ {\rm{)}} 
and $c_{\epsilon}$, where $\epsilon$ is a primitive $2n$th root of unity.
This group is denoted by $Q_7$ {\rm{(}}it is isomorphic 
to $BD_{4n} \rtimes (d_1)$,
the semidirect product of the binary dihedral group of order $4n$ 
{\rm{(}}generated by $c_\epsilon$ and $s${\rm{)}} and the cyclic
group of order $2$ {\rm{(}}generated by $d_1${\rm{)}}{\rm{)}}.
\item[(4)]
The group of order $8n$ generated by $s$
and $c_{\epsilon, -}$, where $\epsilon$ is a primitive $4n$th
root of unity. This group is denoted by $Q_8$ {\rm{(}}it is isomorphic to
a semidirect product of two cyclic groups $C_{4n} \rtimes C_2$,
as $s c_{\epsilon, -} s = c_{\epsilon,-}^{2n-1}${\rm{)}}.
\item[(4')]
The group of order $8n$ generated by $s_1$
and $c_{\epsilon, -}$, where $\epsilon$ is a primitive $4n$th
root of unity. This group is conjugate to the group in case {\rm{(4)}}, 
so is also denoted by $Q_8$.
\end{enumerate}
\end{lemma}

Note that $Q_5$ appeared in case $(D_n)$, $Q_6$ in case $(A_{n,5})$.
We will see that groups $Q_7$ and $Q_8$ arise in the proof below.
Although they have the same order, the groups $Q_7$ and $Q_8$ are 
not isomorphic (e.g. when $n=1$, $Q_7$ is the dihedral group 
generated by $d_1s$ and $s$ (here $c_\epsilon = (d_1s)^2$), while 
$Q_8$ is the direct product of two cyclic groups $C_4 \times C_2$).

\begin{proof}[Proof of Lemma \ref{zzlem3.5}]
First note that the group generated by $s_1$ and $Q_4$
is conjugate to the group generated by $s$ and $Q_4$
using the element
$$\begin{pmatrix}
\epsilon & 0\\
0 & 1
\end{pmatrix}.$$
Hence the groups in case (4) and case (4') are isomorphic.

Pick $g\in H\setminus O$. By \eqref{E3.3.1}, $H$ is generated by
$H\cap O$ and $g$. By \eqref{E3.0.1}, $\det g=\pm 1$. There
are two cases to consider.

Case (i): Suppose there is a $g\in H\setminus O$ such that $\det g=1$.
Then, up to  conjugation, $g=s_1$. Since $s_1^2=-{\mathbb I}$,
$-{\mathbb I}\in H\cap O$. By Lemma \ref{zzlem3.4}, $H\cap O$ is
either $Q_1,Q_2$ or $Q_4$. (Note that $Q_3$ cannot occur as $-{\mathbb I}
\not\in Q_3$.) Therefore $H$ is of cases (1), (3) or (4'), respectively.

Case (ii): Suppose $\det g=-1$ for all $g\in H\setminus O$.
Then, up to  conjugation, $g=s$. We may also assume that $s_1
\not\in H$ (otherwise this is case (i)). Since $s_1=s d_1$,
$d_1\not\in H\cap O$. By Lemma \ref{zzlem3.4}, $H\cap O$ is either
$Q_1$ or $Q_4$. Therefore $H$ occurs in case (2) or (4) respectively.
Finally case (4) cannot occur as $\det (s c_{\epsilon,-})=1$.
\end{proof}

We summarize the groups we have found in the table below:
$$\begin{array}{|c|c|c|}
\hline
 H & \text{ generators}& \epsilon \text{ primitive root} \\
 \hline
 Q_1 & c_\epsilon &  nth \text{ root }\\
 Q_2 & d_1, \quad c_\epsilon &  2nth \text{ root }\\
 Q_3 & d_1, \quad c_\epsilon &   \text{ odd root }\\
 Q_4 & c_{\epsilon,-}  & 4nth \text{ root }\\
 Q_5 & s_1, \quad  c_\epsilon &  2nth \text{ root }\\
 Q_6 & s, \quad c_\epsilon & nth \text{ root }\\
 Q_7 & d_1, \quad s, \quad c_\epsilon  &  2nth \text{ root }\\
 Q_8 & s, \quad c_{\epsilon,-} &  4nth \text{ root }\\
\hline
\end{array}$$

\begin{center}
Table 3: Finite subgroups of $U$ that satisfy sequence \eqref{E3.0.1}
\end{center}

\section{Involutions on ADE singularities, noncommutative case}
\label{zzsec4}
In this section an analysis similar to that in the last section, 
but for the noncommutative cases that are needed later, is summarized 
briefly. Let $A$
be a noncommutative AS regular algebra of global dimension two that is
generated in degree 1. Let $H$ be a finite subgroup of $\Aut(A)$ which
satisfies the following exact sequence
\begin{equation}
\label{E4.0.1}\tag{E4.0.1} 1\to G\to
H\xrightarrow{\hdet} \{\pm 1\}\to 1
\end{equation}
where $G=H\cap \AutSL(A)$. The following lemma is similar to Lemma
\ref{zzlem3.2} (with part ($A_{n,5}$) not occurring) and Lemma
\ref{zzlem3.4}. Its proof is omitted.

\begin{lemma}
\label{zzlem4.1}
Let $A=k_{q}[t_1,t_2]$ and $H$ be a finite subgroup of $\Aut(A)$
satisfying \eqref{E4.0.1}. Suppose either
\begin{enumerate}
\item[(i)]
$q\neq \pm 1$ {\rm{(}}and whence $H\subset O${\rm{)}}, or
\item[(ii)]
$q=-1$ and $H\subset O$.
\end{enumerate}
Then, up to a permutation of $\{t_1,t_2\}$, $H$ is equal to one of
the following groups.
\begin{enumerate}
\item[($A^q_{n,1}$)]
$H=\langle {\mathbb I}, d_1\rangle=C_2$.
\item[($A^q_{n,2}$)]
$H$ is generated by $d_1$ and $c_{\epsilon}$, where $\epsilon$ is a
primitive $2n$th root of unity, and $H=C_{2n}\times C_2$.
\item[($A^q_{n,3}$)]
$H$ is generated by $d_1$ and $c_{\epsilon}$, where $\epsilon$ is a
primitive $n$th root of unity and $n$ is an odd integer. In this case,
$H=C_{n}\times C_2\cong C_{2n}$.
\item[($A^q_{n,4}$)]
$H$ is generated by $c_{\epsilon,-}$, where $\epsilon$ is a
primitive $4n$th root of unity, and $H=C_{4n}$.
\end{enumerate}
\end{lemma}

Next we compute the fixed subrings.  
Let $S=k_q[t_1,t_2]^G$ and let $g=d_1\in H\setminus G$ in the first
three cases, and $g=c_{\epsilon, -}\in H\setminus G$ in case ($A^q_{n,4}$).
Let $h$ be the induced automorphism of $g$ on $S$.

\begin{lemma}
\label{zzlem4.2} Retaining the hypotheses of Lemma \ref{zzlem4.1} and 
setting,
$$S=k_q[t_1,t_2]^G = k_{p_{ij}}[X_1,X_2,X_3]/(f).$$
In cases $(A^q_{n,1})$ and $(A^q_{n,2})$, $h$ lifts to a
reflection of $k_{p_{ij}}[X_1,X_2,X_3]$ and preserves the relation
$f$.
\end{lemma}

\begin{proof}
We consider each of the cases:

\begin{enumerate}
\item[($A^q_{n,1}$):]
$S=k_q[t_1, t_2]$ and $h=g: t_1\to -t_1, t_2\to t_2$ is a reflection
of $S$, and it lifts to $h: X_1 \to -X_1, X_2 \to X_2, X_3 \to X_3$ 
and preserves $f=X_3$.
\item[($A^q_{n,2}$):]
Let $X_1:=t_1^{2n} , X_2:=t_2^{2n}$ and $X_3:=t_1t_2$. Then
$$S=k_{p_{ij}}[X_1, X_2, X_3]/(X_1X_2-q^{-{2n \choose 2}}X_3^{2n})$$ where
$p_{12}=q^{4n^2}$, $p_{13}=q^{2n}$ and $p_{23}=q^{-2n}$. It is easy to
see that
$$h: X_1\to X_1, \quad X_2\to X_2, \quad X_3\to -X_3.$$
Note that $h$ lifts to a reflection of $k_{p_{ij}}[X_1,X_2,X_3]$ and
preserves\\ $f=X_1X_2-q^{-{2n \choose 2}}X_3^{2n}$.
\item[($A^q_{n,3}$):]
Let $X_1:=t_1^n , X_2:=t_2^{n}$ and $X_3:=t_1t_2$
then $$S=k_{p_{ij}}[X_1, X_2, X_3]/(X_1X_2-q^{-{n \choose 2}}X_3^{n}),$$
where $n$ is an odd integer, and
where $p_{12}=q^{n^2}$, $p_{13}=q^{n}$ and $p_{23}=q^{-n}$. Then
$$h: X_1\to -X_1, \quad X_2\to X_2, \quad X_3\to -X_3.$$
In this case, $h$ does not lift to a reflection of $k_{p_{ij}}[X_1,X_2,X_3]$
and does not preserve $f=X_1X_2-q^{-{n \choose 2}}X_3^{n}$.
\item[($A^q_{n,4}$):]
Let $X_1:=t_1^{2n} , X_2:=t_2^{2n}$ and $X_3:=t_1t_2$ then
$$S=k[X_1, X_2, X_3]/(X_1X_2-q^{-{2n \choose 2}}X_3^{2n}),$$ where
$p_{12}=q^{4n^2}$, $p_{13}=q^{2n}$ and $p_{23}=q^{-2n}$, and
$$h: X_1\to -X_1, \quad X_2\to -X_2, \quad X_3\to -X_3.$$
In this case, $h$ preserves $f=X_1X_2-q^{-{2n \choose 2}}X_3^{2n}$,
but does not lift to a reflection of $k_{p_{ij}}[X_1,X_2,X_3]$.
\end{enumerate}
\end{proof}


\section{Definitions of Complete intersection and bireflections}
\label{zzsec5}
In this section we recall the definitions of complete intersection,
bireflection and cyclotomic Gorenstein in the noncommutative setting, 
which were proposed in \cite{KKZ4} . A set of (homogeneous) elements
$\{\Omega_1,\cdots,\Omega_n\}$ in a graded algebra $C$ is called
{\it a sequence of regular normalizing elements} if, for each $i$, the
image of $\Omega_i$ in $C/(\Omega_1,\cdots,\Omega_{i-1})$ is regular
(i.e., a non-zero-divisor) and normal.

\begin{definition}
\label{zzdef5.1}\cite{KKZ4}
Let $A$ be a connected graded noetherian algebra.
\begin{enumerate}
\item[(cci)]
We say $A$ is a {\it classical complete intersection ring} (or {\it cci}
for short) if there is a connected graded noetherian AS regular
algebra $C$ and a sequence of regular normal elements $\{\Omega_1,
\cdots,\Omega_n\}$ such that $A$ is isomorphic to $C/(\Omega_1,
\cdots,\Omega_n)$.
\item[(gci)]
We say $A$ is a {\it complete intersection ring of type GK} (or {\it
gci} for short) if the $\Ext$-algebra $\Ext^*_{A}(k,k)$ has finite
GK-dimension in the sense of Definition \ref{zzdef1.1}.
\end{enumerate}
\end{definition}

In the commutative case, these two definitions of complete intersection
are equivalent, and are equivalent to the following condition, see
\cite[Theorem C]{FHT} and \cite{Gu},
\begin{enumerate}
\item[(nci)]
the $\Ext$-algebra $\Ext^*_{A}(k,k)$ is noetherian.
\end{enumerate}

In the noncommutative case the condition (nci) is not equivalent to
either (cci) or (gci), but (cci) implies (gci). We refer to 
\cite{KKZ4} for more details and examples.

By \eqref{E1.1.1}, the GK-dimension of the $\Ext$-algebra
$\Ext^*_A(k,k)$ is given by
$$\GKdim \Ext^*_A(k,k)=\limsup_{n\to\infty}
\frac{\log (\sum_{i=0}^n \dim \Ext_{A}^i(k,k))}{\log n}.$$
We can define a numerical invariant of the algebra $A$, called the
{\it gci number} of $A$, by
$$gci(A)=\GKdim \Ext^*_A(k,k).$$
Since $A$ is noetherian, by using the minimal free resolution of the
trivial $A$-module $k$, one can check that $\Tor^A_i(k,k)^*\cong
\Ext^i_A(k,k)$ \cite[Proposition 5.1, p. 120]{CE}. Thus we have
\begin{equation}
\label{E5.1.1}\tag{E5.1.1}
gci(A)=\limsup_{n\to\infty}
\frac{\log (\sum_{i=0}^n \dim \Tor^{A}_i(k,k))}{\log n}.
\end{equation}

A related notion is the following.

\begin{definition} \cite{KKZ4}
\label{zzdef5.2}
Let $A$ be a connected graded noetherian algebra.
\begin{enumerate}
\item
We say $A$ is {\it cyclotomic} if its Hilbert series $H_A(t)$ is
of the form $p(t)/q(t)$, where $p(t)$ and $q(t)$ are coprime
integral polynomials, and if all the roots of $p(t)$ are roots
of unity. Note that all roots of $q(t)$ are roots of
unity since $A$ is noetherian.
\item
We say $A$ is {\it cyclotomic Gorenstein} if $A$ is both
cyclotomic and AS Gorenstein.
\end{enumerate}
\end{definition}

We review some results proved in \cite{KKZ4}.

\begin{lemma}
\label{zzlem5.3}
\cite{KKZ4}
Let $A$ be a connected graded noetherian ring.
\begin{enumerate}
\item
If $A$ is a cci, then it is a gci.
\item
Suppose that $A$ is isomorphic to $R^G$ for some noetherian
Auslander regular algebra $R$ and a finite group
$G\subset \Aut(R)$. If $A$ is a gci, then it is cyclotomic Gorenstein.
\end{enumerate}
\end{lemma}

The notion of cyclotomic Gorenstein is strictly weaker than the notion
of complete intersection, even in the commutative case (see
Stanley \cite[Example 3.9]{St1}). However, Stanley's example
is not a fixed subring of an AS regular algebra.

The next two easy lemmas will be used later. Let $A= C[t;\sigma]$
denote the skew polynomial Ore extension of $C$ by an automorphism $\sigma$
of $C$ \cite[Section 1.2.3]{MR}.

\begin{lemma}
\label{zzlem5.4} Let $C$ be a noetherian connected graded algebra,
and $G$ be a finite subgroup of $\Aut(C)$ such that $C^G$ is a
cci {\rm{(}}respectively, gci{\rm{)}}.
\begin{enumerate}
\item
Let $\sigma\in \Aut(C)$ be an automorphism that commutes with all 
$g\in G$. By abuse of
notation, let $G$ also denote the finite subgroup of $\Aut(C[t;\sigma])$
induced from $G$, with $g(t)=t$ for all $g\in G$ and $g\mid_C=g$. Then
$(C[t;\sigma])^G$ is a cci {\rm{(}}respectively, gci{\rm{)}}.
\item
Let $\{\Omega_1, \cdots, \Omega_n\}$ be a sequence of regular normalizing
elements in $C$ such that $g(\Omega_i)=\Omega_i$ for all $g\in G$ and
all $i$. Let $B=C/(\Omega_1,\cdots,\Omega_n)$ and let $G$ also denote
the subgroup of $\Aut(B)$ induced from $G$. Then $B^G\cong C^G/
(\Omega_1,\cdots,\Omega_n)$.
\item
Under the hypotheses of part {\rm{(2)}} $B^G$ is a cci 
{\rm{(}}respectively, gci{\rm{)}}.
\end{enumerate}
\end{lemma}

\begin{proof} (a) Since $(C[t;\sigma])^G=C^G[t;\sigma]$, the assertion
follows.

(2) First we claim that, if $\Omega$ is a regular normal
element of an algebra $D$ and $g(\Omega)=\Omega$ for all
$g\in G$, then $\Omega$ is a regular normal element in $D^G$
and $D^G/(\Omega)=(D/\Omega)^G$. To see
this, we use the short exact sequence
$$0\to D[\deg \Omega]\xrightarrow{r_{\Omega}} D\to D/(\Omega)\to 0.$$
Applying the Reynolds operator $\int:=\frac{1}{|G|}
\sum_{g\in G} g$ to the sequence above, we have
$$0\to D^G[\deg \Omega]\xrightarrow{r_{\Omega}} D^G
\to (D/(\Omega))^G\to 0.$$
Therefore $\Omega\in D^G$ is a regular normal element, and
$D^G/(\Omega)=(D/\Omega)^G$. The assertion follows by induction.

(3) By induction on $n$ we need to show the assertion only for $n=1$.
Write $\Omega=\Omega_1$. Clearly $\Omega$ is a regular normal element
in $C^G$. If $C^G$ is a cci (respectively, gci), so is $C^G/(\Omega)$.
By part (2), $B^{G}\cong C^G/(\Omega)$, and the assertion follows.
\end{proof}

\begin{lemma}
\label{zzlem5.5} Let $B$ be a cci of the
form $C/(\Omega_1,\cdots,\Omega_n)$, where $C$ is an AS regular algebra
and $\{\Omega_1,\cdots,\Omega_n\}$ is a sequence of regular normalizing
elements of $C$.  Suppose that $\sigma \in \Aut(B)$ is a graded 
automorphism of $B$ that lifts 
to a graded automorphism $\sigma'\in \Aut(C)$.
Let $A=B[v;\sigma]$, and $\deg v=d$. Suppose that $h \in \Aut(A)$ is 
such that $h(v) = -v$ and $g := h\mid_B$ is a graded automorphism of $B$,
where $g$ lifts to a graded automorphism $g' \in \Aut(C)$.
 Assume that:
\begin{enumerate}
\item there exists a set
$\{x_1,\cdots,x_m\}$ of degree 1 elements that generate the algebra $C$,
with $x_m$ a regular normal element in $C$;
\item
$g'\in \Aut(C)$ is an involutory reflection of $C$ such that the fixed
subring $C^{g'}$ is AS regular, $g'(x_j)=x_j$ for $j<m$, and
$g'(x_m)=-x_m$;
\item
$\sigma'(x_m)=\lambda x_m$ for some $\lambda\in k^\times$,
\item
$g'(\Omega_i)=\Omega_i$ for all $i=1,\cdots, n$,
\item $\sigma'$ and $g'$ commute.
\end{enumerate}
Then the fixed subring $A^h$ is a cci.
\end{lemma}

\begin{proof} By Lemma \ref{zzlem5.4}(2),
$$(C/(\Omega_1\cdots,\Omega_n))^g=C^{g'}/(\Omega_1,\cdots,\Omega_n)$$
where $\{\Omega_1,\cdots,\Omega_n\}$ is a sequence of regular normalizing
elements in $C^{g'}$.

By the hypotheses, $C^{g'}$ is AS regular. Since $x_m$ is a normal element
in $C$, $C^{g'}$ is generated by $\{x_1,\cdots,x_{m-1}, x_{m}^2\}$.
Formally set $x'_m:=x_m^2$, $y_m := x_m v \in C[v;\sigma']$ and $v_2:=v^2$. 
Suppose that
conjugation by $x_m$ is the graded automorphism $\sigma_m'$ of $C$, i.e. 
$x_mc = \sigma'(c)x_m$ for
all $c \in C$.  Then for any $c \in C$ we have $y_m
c=\sigma_m'(\sigma'(c)) y_m$, $v_2 c= \sigma'^2(c) v_2$, and $v_2
y_m=\lambda^2 y_m v_2$.  The map $\sigma_{m}'\circ\sigma'$ is an
automorphism of $C^{g'}$ since for $c \in C^{g'}$ we have 
$x_m \sigma'(c) = \sigma_m'(\sigma'(c))x_m$, so
applying $g'$ gives $-x_m g'(\sigma'(c)) = g'(\sigma_m'(\sigma'(c)) )(-x_m)$ or
$x_mg'(\sigma'(c)) = g'(\sigma_m'(\sigma'(c))x_m$, so
that $x_m(\sigma'(g'(c))) = x_m\sigma'(c) = \sigma_m'(\sigma'(c))x_m 
= g'(\sigma_m'(\sigma'(c)))x_m $,
and hence $\sigma_m'(\sigma'(c)) = g'(\sigma_m'(\sigma'(c))) $.
Let $W$ be the iterated Ore extension
$$W:=C^{g'}[Y_m;\sigma_{m}'\circ\sigma'][V_2; \sigma'^2],$$ where we use $Y_m$
and $V_2$ as independent variables of degrees $d + \text{deg } x_m$ 
and $2d$, respectively. Then $W$ is AS regular.

We claim that $A^h\cong W/ (\Omega_1,\dots,\Omega_n, Y_m^2-\lambda
x'_{m}V_2)=:A'$, which is a cci by definition. First, there is an 
algebra surjective map from $A'$ to $A^h$ sending
$x_i\to \overline{x_i}$ for $i<m$, $x_m^2\to \overline{x_m}^2=: 
\overline{x'_m}$, $Y_m\to \overline{x_m} v=:\overline{y_m}$
and $V_2\to v^2=:v_2$, where bar indicates the image in $A$. Since 
$x_m$ is a normal element, there is a short exact sequence
\begin{equation}
\label{E5.5.1}\tag{E5.5.1} 0\to C[\deg x_m]\xrightarrow{r_{x_m}}
C\to C/(x_m)\to 0, \end{equation} which implies that the Hilbert
series of $C$ is $H_C(t)=H_{C/(x_m)}(t)(1-t^{\deg x_m})^{-1}$.
Applying $g'$ to \eqref{E5.5.1}, it follows that the trace
of $g'$ is
$$Tr_C(g',t)= H_{C/(x_m)}(t)\; \frac{1}{1+t^{\deg x_m}}.$$
An argument similar to the proof of Lemma \ref{zzlem1.6} shows that
$$Tr_A(h,t)=H_{C/(x_m)}(t)\frac{\prod_{i=1}^n (1-t^{\deg \Omega_i})}
{(1+t^{\deg x_m})(1+t^d)},$$
where $d=\deg v$, and the Hilbert series of $A$ is
$$H_A(t)=H_{C/(x_m)}(t)\frac{\prod_{i=1}^n (1-t^{\deg \Omega_i})}
{(1-t^{\deg x_m})(1-t^d)}.$$ By Molien's 
theorem, $H_{A^h}(t) = (H_{A}(t) + Tr_A(h,t))/2$, so
\begin{equation}\label{E5.5.2}\tag{E5.5.2}
H_{A^h}(t)=H_{C/(x_m)}(t)\frac{(1+t^{d+\deg x_m})\prod_{i=1}^n
(1-t^{\deg \Omega_i})} {(1-t^{2\deg x_m})(1-t^{2d})}.
\end{equation}
By comparing their Hilbert series, one observes that $C/(x_m)\cong
C^{g'}/(x_m^2)$. Then
$$\begin{aligned}
H_{A^h}(t) &=H_{C/(x_m)}(t)\frac{(1+t^{d+\deg x_m})\prod_{i=1}^n (1-t^{\deg
\Omega_i})} {(1-t^{2\deg
x_m})(1-t^{2d})}\\
&=H_{C^{g'}/(x_m^2)}(t)\frac{(1+t^{d+\deg x_m})\prod_{i=1}^n
(1-t^{\deg \Omega_i})} {(1-t^{2\deg x_m})(1-t^{2d})}
\\
&=H_{C^{g'}}(t)\frac{(1+t^{d+\deg x_m})\prod_{i=1}^n (1-t^{\deg
\Omega_i})} {(1-t^{2d})}\\
&=H_{C^{g'}}(t)\frac{(1-t^{2d+2\deg x_m})\prod_{i=1}^n (1-t^{\deg
\Omega_i})}
{(1-t^{2d})(1-t^{d+\deg x_m})}\\
&=H_{A'}(t).
\end{aligned}
$$
Therefore $A^h\cong A'$, completing the proof.
\end{proof}

A result of Kac-Watanabe \cite{KW} and Gordeev \cite{G1} states that
if $G$ is a finite subgroup of $GL_n(k)$ acting naturally on the
commutative polynomial ring $k[x_1, \cdots, x_n]$ and the fixed
subring $k[x_1, \cdots, x_n]^G$ is a complete intersection, then $G$
is generated by bireflections (elements such that $\rank(g-I) \leq
2$). This result establishes a connection between ``complete intersection''
and ``bireflection''. We would like to extend this connection to the
noncommutative case. First we need to understand what a
bireflection in the noncommutative setting is.

\begin{definition}
\label{zzdef5.6}\cite{KKZ4} Let $A$ be a noetherian connected graded
algebra of GK-dimension $n$.  We call a graded automorphism $g$ of
$A$ a {\it bireflection} if its trace has the form:
$$Tr_A(g,t) = \frac{p(t)}{(1-t)^{n-2} q(t)},$$
where $p(t)$ and $q(t)$ are integral polynomials with $p(1)q(1) \neq
0$.

As in the classical case, it is convenient to view a reflection in 
the sense of Definition
\ref{zzdef1.7} as a bireflection, also.
\end{definition}

Note that a bireflection is called a quasi-bireflection in
\cite{KKZ4}. At this point we are not able to prove a version
of the Kac-Watanabe and Gordeev theorem for arbitrary
AS regular algebras.  In this paper we verify a version of the theorem
for a class of noncommutative algebras. We make the following
definitions.

\begin{definition}
\label{zzdef5.7}
Let $A$ be a noetherian connected graded algebra that is a gci.
Let $\Gamma$ be a subgroup of $\Aut(A)$. Let $G$ be a finite
subgroup of $\Gamma$ and consider the following four conditions:
\begin{enumerate}
\item[(i)]
$A^G$ is a classical complete intersection.
\item[(ii)]
$A^G$ is a complete intersection of GK type.
\item[(iii)]
$A^G$ is cyclotomic Gorenstein and $G$ is generated by bireflections .
\item[(iii')]
$A^G$ is cyclotomic Gorenstein.
\end{enumerate}
Then
\begin{enumerate}
\item
We say $A$ is {\it $\Gamma$-Gordeev} if (ii) implies (iii) for
every finite subgroup $G\subset \Gamma$.
\item
The algebra $A$ is said to be {\it $\Gamma$-Kac} (respectively,
{\it strongly $\Gamma$-Kac}) if (iii) (respectively, (iii'))
implies (ii) for every finite subgroup $G\subset \Gamma$.
\item
If (ii) and (iii) are equivalent for every finite subgroup
$G\subset \Gamma$, then $A$ is called {\it $\Gamma$-Gordeev-Kac}.
\item
If (i), (ii) and (iii) are equivalent for every finite subgroup
$G\subset \Gamma$, then $A$ is called {\it $\Gamma$-Gordeev-Kac-Watanabe}
(or {\it $\Gamma$-GKW}).
\item
Finally, if (i), (ii), (iii) and (iii') are equivalent for every finite
subgroup $G\subset \Gamma$, then $A$ is called {\it strong
$\Gamma$-Gordeev-Kac-Watanabe} (or {\it strong $\Gamma$-GKW}).
\end{enumerate}
If $\Gamma$ is the whole group $\Aut(A)$, then the prefix ``$\Gamma$-''
is omitted.
\end{definition}

Using the definition, the commutative polynomial ring
$k[x_1, \cdots, x_n]$ is Gordeev by the Kac-Watanabe and Gordeev theorem.
Theorem \ref{zzthm0.1} implies that AS regular algebras of dimension 2
are strong GKW. Theorem \ref{zzthm0.3} says that the down-up algebras
are Gordeev-Kac.

\begin{lemma}
\label{zzlem5.8} Let $A$ be a noetherian connected graded domain of 
finite GK-dimension, and
let $G$ be a finite subgroup of $\Aut(A)$. Let $R$ be the subgroup of
$G$ generated by bireflections in $G$. Then, for every $h\in
G\setminus R$, the induced map $h'$ of $h$ on $A^R$ is not a
bireflection.
\end{lemma}

\begin{proof} Since $R$ is normal in $G$, $h: A^R \to A^R$.
By \cite[Lemma 5.2]{JiZ},
\begin{equation}
\label{E5.8.1}\tag{E5.8.1} Tr_{A^R}(h',t)=\frac{1}{|R|}\sum_{g\in R}
Tr_A(hg, t). \end{equation} Let $n=\GKdim A$. Since $hg$ is not a
bireflection, $(1-t)^{n-3}Tr_A(hg, t)$ is analytic at $t = 1$ for all
$g\in R$. Hence \eqref{E5.8.1} implies that $(1-t)^{n-3}Tr_{A^R}(h', t)$
is analytic at $t = 1$, and therefore $h'$ is not a bireflection of $A^R$.
\end{proof}

We finish this section with an example of bireflections and
noncommutative complete intersections that will be used
in Section \ref{zzsec8}.

\begin{example}
\label{zzex5.9} Let $q$ be a nonzero scalar. Let $T_q$ be a skew
polynomial ring generated by $x_1,x_2$ and $y_1,y_2$, where
$\deg x_1=\deg x_2=v>0$ and $\deg y_1=\deg y_2=w>0$, subject to the
relations
$$\begin{aligned}
x_2x_1&=x_1x_2,\\
y_1x_1&=q x_1 y_1,\\
y_1x_2&=q^{-1} x_2y_1,\\
y_2x_1&=q^{-1} x_1 y_2,\\
y_2x_2&=q x_2y_2,\\
y_2y_1&=y_1y_2.
\end{aligned}
$$
Let $\zeta_1, \zeta_2\in \Aut(T_q)$ be defined by
$$\zeta_1: x_1\to x_1, x_2\to x_2, y_1\to -y_1, y_2\to - y_2,$$
and
$$\zeta_2: x_1\to x_2, x_2\to x_1, y_1\to c y_2, y_2\to c^{-1} y_1$$
for some $c\in k^\times$. By choosing a new $y_2$, we may assume $c=1$.

{\bf Case (i)}: We claim that the fixed subring $T_q^{\zeta_1}$ is a cci.
It is easy to see that
$$Tr_{T^q}(\zeta_1,t)=\frac{1}{(1-t^v)^2(1+t^w)^2},$$
so $\zeta_1$ is a bireflection. The fixed subring $T_q^{\zeta_1}$ is isomorphic
to $C/(\Theta)$, where $C$ is the skew polynomial ring
generated by $x_1,x_2$, $Y_1:=y_1^2,Y_2:=y_2^2$ and $Y_3:=y_1y_2$
subject to the
relations
$$\begin{aligned}
x_2x_1&=x_1x_2,\\
Y_1x_1&=q^2 x_1 Y_1,\\
Y_1x_2&=q^{-2} x_2Y_1,\\
Y_2x_1&=q^{-2} x_1 Y_2,\\
Y_2x_2&=q^2 x_2Y_2,\\
Y_2Y_1&=Y_1Y_2,\\
Y_3x_1&=x_1Y_3,\\
Y_3x_2&=x_2Y_3,\\
Y_3Y_1&=Y_1Y_3,\\
Y_3Y_2&=Y_2Y_3,
\end{aligned}
$$
and $\Theta:=Y_1Y_2-Y_3^2$. Therefore $T_q^{\zeta_1}$ is a cci.
Note that $C\cong T_{q^2}[Y_3]$ as ungraded algebras.

{\bf Case (ii)}: Next we prove that the fixed subring $T_q^{\zeta_2}$ is a cci.
Since $T_q$ is a ``twisted tensor product'' of $k[x_1,x_2]$
with $k[y_1,y_2]$, $T_q\cong k[x_1,x_2] \otimes k[y_1,y_2]$ as
graded vector spaces, we can compute the trace of $\zeta_2$ by
computing the trace of the action of $\zeta_2$ on $k[x_1,x_2]$
and $k[y_1,y_2]$ separately. Hence,
$$Tr_{T_q}(\zeta_2,t)=\frac{1}{(1-t^{2v})(1-t^{2w})},$$
and $\zeta_2$ is a bireflection of $T_q$.  Molien's Theorem shows 
that the fixed subring has Hilbert series
$$H_{T_q^{\zeta_2}}(t)= 
\frac{1-t^{2(v+w)}}{(1-t^v)(1-t^w)(1-t^{2v})(1-t^{2w})(1-t^{v+w})}.$$

If $q=\pm 1$, using $H_{T_q^{\zeta_2}}(t)$  one sees that the fixed 
subring $T_q^{\zeta_2}$ is isomorphic
to $C/(\Theta)$, where $C$ is the skew polynomial ring
generated by
$$X_1:=x_1+x_2, \; X_2:=x_1x_2,\; Y_1:=y_1+y_2,\; Y_2:=y_1y_2,$$
$$Z_1:=y_1x_1+y_2x_2,$$
and subject to the relations
$$\begin{aligned}
X_2X_1&=X_1X_2,\\
Y_1X_1&=q X_1 Y_1,\\
Y_1X_2&=X_2Y_1,\\
Y_2X_1&=X_1Y_2,\\
Y_2X_2&=X_2Y_2,\\
Y_2Y_1&=Y_1Y_2,\\
Z_1X_1&=q X_1 Z_1,\\
Z_1X_2&=X_2 Z_1,\\
Z_1Y_1&=qY_1 Z_1,\\
Z_1Y_2&=Y_2 Z_1,
\end{aligned}
$$
and
$$\Theta:=(X_1Y_1-2 q Z_1)^2- q (X_1^2-2X_2)(Y_1^2-2Y_2)- 8 q  X_2Y_2 
+ 2 q X_2Y_1^2 + 2 q X_1^2Y_1$$
is a central element of $C$. Therefore
$T_{q=\pm 1}^{\zeta_2}$ is a cci.

If $q\neq \pm 1$, we claim that again Molien's Theorem can be used 
to show that the fixed subring $T_q^{\zeta_2}$ is isomorphic
to $C/(\Theta_1,\Theta_2)$, where $C$ is the AS regular algebra
generated by
$$X_1:=x_1+x_2, \; X_2:=x_1x_2,\; Y_1:=y_1+y_2,\; Y_2:=y_1y_2,$$
$$Z_1:=q(x_1y_1+x_2y_2)=y_1x_1+y_2x_2, \qquad
Z_2:=x_1y_2+x_2y_1=q(y_2x_1+y_1x_2),$$
and subject to the relations
$$\begin{aligned}
X_2X_1&=X_1X_2,\\
Y_1X_1&= X_1 Y_1+(q^{-1}-1)(Z_2-Z_1),\\
Y_1X_2&=X_2Y_1,\\
Y_2X_1&=X_1Y_2,\\
Y_2X_2&=X_2Y_2,\\
Y_2Y_1&=Y_1Y_2,\\
Z_1X_1&=q X_1 Z_1+q(q^{-1}-q)X_2Y_1,\\
Z_1X_2&=X_2 Z_1,\\
Z_1Y_1&=q^{-1}Y_1 Z_1+(q-q^{-1})X_1Y_2,\\
Z_1Y_2&=Y_2 Z_1,\\
Z_2X_1&=q^{-1}X_1Z_2+(q-q^{-1})X_2Y_1,\\
Z_2X_2&=X_2Z_2,\\
Z_2Y_1&=qY_1Z_2+(1-q^2)X_1Y_2,\\
Z_2Y_2&=Y_2Z_2,\\
Z_2Z_1&=Z_1Z_2+(1-q^2)X_1^2Y_2+(q^2-1)X_2Y_1^2,
\end{aligned}
$$
with a central regular sequence $\{\Theta_1, \Theta_2 \}$, where
$$\Theta_1:=Y_1X_1-q^{-1}X_1Y_1-(1-q^{-2})Z_1$$ and
$$\Theta_2:=Z_2Z_1-q^{-2}Z_1Z_2-(q^2-q^{-2})X_2(Y_1^2-2Y_2)$$
$$=(1-q^{-2})Z_1Z_2 +(1-q^2)X_1Y_2 + (q^{-2}-1)X_2Y_1^2 
+ 2(q^2 - q^{-2})X_2Y_2.$$

First one checks that the associated elements of $T_q$ satisfy the
fifteen relations in $C$ and the additional relations $\Theta_1$ and 
$\Theta_2$.
Next one checks the Diamond Lemma to see that $X_1, X_2, Y_1, Y_2, Z_1$ 
and $Z_2$ form
a PBW basis for $C$.   Then we filter $C$ with the filtration ${F}$, 
where degree $Z_i = 5$ and degree $X_i$ = degree $Y_i = 3$,
so that there is an algebra surjection from $\gr_F C $ onto the skew 
polynomial ring $D$, where $D$ is
an algebra of the form $k_{p_{i,j}}[w_1, \dots, w_6]$, which is an AS 
regular domain of dimension 6.  The PBW basis for $C$ shows that 
there is a vector space isomorphism from
 $\gr_F C$ to $D$, and hence the algebra map from $\gr_F C$ to $D$ 
is an isomorphism.  As $\gr _F C$ is an AS regular domain of 
dimension 6, it follows that $C$ is an AS regular domain of dimension 6.

It can be checked that $\Theta_1$ and $\Theta_2$ are central elements 
of $C$.  As $C$ is a domain, $\Theta_1$ is a regular element of $C$.
Now using the filtration $\cal{F}$, where degree $X_i = $ degree 
$Y_i = $  degree $Z_i = 1$, we see that
$$  \gr_{\cal{F}}(\frac{C} {(\Theta_1)}) \cong 
\frac{\gr_{\cal{F}} C}{(\gr_{\cal{F}}\Theta_1)} \cong
\frac{k[X_1,Y_1]}{(X_1Y_1)}\langle Y_1,Y_2,Z_1,Z_2 \rangle,$$
is an iterated Ore extension, where
$\gr_{\cal{F}} \Theta_2$ is a regular element of 
$\gr_{\cal{F}} (C/(\Theta_1)) $, and hence $\Theta_2$ is a regular 
element of $ C/(\Theta_1) $.  It follows that $\{ \Theta_1, \Theta_2 \}$ 
is a regular central sequence of $C$, and
therefore $T_q^{\zeta_2}$ is a cci.

{\bf Case (iii)}: Let $G=\langle \zeta_1,\zeta_2\rangle$. We show that 
$T_q^G$ is a cci.
Note that $G\cong C_2\times C_2$. So $T_q^{G}=(T_q^{\zeta_1})^{\phi_2}$,
where $\phi_2$ is the induced automorphism of $\zeta_2$ in
$T_q^{\zeta_1}$.
By case (i), $T_q^{\zeta_1}\cong (T_{q^2}[Y_3])/(\Omega)$. Since
$\phi_2$, as an automorphism of $T_{q^2}[Y_3]$, preserves $Y_3$ and
$\Omega$, it follows from Lemma \ref{zzlem5.4}(1,2) that it
suffices to show $T_{q^2}^{\phi_2}$ is a cci. Now $\phi_2$ is
a $\zeta_2$ with $c$ replaced by $c^2$ (and $\deg Y_1=\deg Y_2=2w$).
By case (ii), $T_{q^2}^{\phi_2}$ is a cci, and therefore  $T_q^G$ is a cci.
\end{example}

%
%
%

\section{Filtrations}
\label{zzsec6}
In the first part of this section we review the May spectral sequence
\cite{Ma} that is associated to a filtered algebra.
Let $\Lambda$ be a group, usually taken to be ${\mathbb Z}^d$ for some
integer $d$. A $\Lambda$-graded space $X$ is called
$\Lambda$-locally finite if $X=\bigoplus_{g\in \Lambda} X_g$, where
each $X_g$ is finite dimensional over $k$. Assume that
$X:= \bigoplus_{n\in {\mathbb Z}, g\in \Lambda} X^n_g$ is a
${\mathbb Z}\times \Lambda$-graded, $\Lambda$-locally finite, complex
with differential $d$ of degree $(1, 0)\in {\mathbb Z}\times \Lambda$.

\begin{definition}
\label{zzdef6.1} A {\it filtration} on the complex $X$ is a sequence
of ${\mathbb Z}\times \Lambda$-graded subcomplexes of $X$, say $F:=\{
F(n)\mid n\in {\mathbb Z}\}$, such that the following hold:
\begin{enumerate}
\item
$F(n+1)\subset F(n)$ for all $n$,
\item
$\bigcap_{n\in {\mathbb Z}} F(n)=\{0\}$, and
\item
$\bigcup_{n\in {\mathbb Z}} F(n)= X$.
\end{enumerate}
The {\it associated graded complex} is defined to be
$$\gr_F X=\bigoplus_{n\in {\mathbb Z}} F(n)/F(n+1)$$
with the induced differential.
\end{definition}

Later we consider filtrations only on complexes that are concentrated
in degree $\{0\}\times \Lambda$;  in this case, the differential
is trivial. For any $a\in F(n)$,  let $\widehat{a}$
denote the element $a+F(n+1)$ in $F(n)/F(n+1)$. Then $\gr_F X$ is a
${\mathbb Z}^2\times \Lambda$-graded $k$-space where the second
${\mathbb Z}$-grading comes from the filtration. For any $a$, the
differential $d$ is defined to be
$$d(\widehat{a})=\widehat{d(a)}$$
for all $\widehat{a}\in F(n)/F(n+1)$.

Since $\gr_F X$ is a complex, we can compare the cohomologies
$H(\gr_F X)$ with $H(X)$.  Next is a standard lemma
\cite[Theorem 2.6]{Mc}.

\begin{lemma}
\label{zzlem6.2} Retain the above notation. Suppose that $X$ is
$\Lambda$-locally finite. Then there is a convergent spectral
sequence
$$E^{p,q}_1:=H^{p+q}(F(p)/F(p+1))\Longrightarrow H^{p+q}(X).$$
\end{lemma}

\begin{proof} Let $g\in \Lambda$. Restricting to a $g$-homogeneous
component of $X$, we may assume that $X$ is finite dimensional, so
the filtration is bounded. The assertion now follows from
\cite[Theorem 2.6]{Mc}.
\end{proof}

Let $A$ be a ${\mathbb Z}\times \Lambda$-graded algebra which is a dg
(differential graded) algebra with respect to the first grading, and
is a locally finite algebra with respect to the second
grading.

\begin{definition}
\label{zzdef6.3} Let $A$ be a dg algebra.
An {\it algebra filtration} of $A$ is a filtration of
the complex $A$ that satisfies the following extra conditions:
\begin{enumerate}
\item
$k\subset F(0)$,
\item
$F(1)\subset F(0)_{*, \geq 1}$,
\item
$F(n)F(m)\subset F(n+m)$ for all $n,m\in {\mathbb Z}$.
\end{enumerate}
The {\it associated dg algebra} is defined to be
$$\gr_F A=\bigoplus_{n\in {\mathbb Z}} F(n)/F(n+1).$$
\end{definition}

Next we consider a version of the May spectral sequence associated 
to a filtered
algebra; see \cite[Theorem 3]{Ma}. We say $A$ is {\it augmented} if there is
a dg algebra homomorphism $\epsilon: A\to k$. Suppose
$\Lambda={\mathbb Z}^d$. We say a $\Lambda$-graded vector
space $X$ is {\it connected graded} if $X=\bigoplus_{g\in {\mathbb N}^d}
X_g$ and $X_0=k$.

\begin{proposition}
\label{zzpro6.4} 
Let $A$ be a ${\mathbb Z}\times \Lambda$-graded
filtered dg algebra that is $\Lambda$-locally finite and connected graded.
Then there is a convergent spectral sequence
$$\Tor^{\gr A}_{p+q}(k,k)(p)\Longrightarrow \Tor^{A}_{p+q}(k,k).$$
\end{proposition}

\begin{proof} Let $I:=\ker \epsilon$ be the augmented ideal of $A$.
Let $B(A)$ be the bar complex $(T(sI), \partial)$ where $s$ is the
homological degree shift, see \cite[Section 3]{Ma}.
Then $\Tor^A_n(k,k)$ can be computed by \cite[(3.4)]{Ma} as follows
$$\Tor^A_n(k,k)=H^n(B(A)).$$
Since $A$ is filtered, so is $B(A)$ by the induced filtration on $I$.
Since $A$ is $\Lambda$-locally finite and connected graded,
so is $B(A)$. By Lemma \ref{zzlem6.2}, there is a convergent
spectral sequence
$$H^{p+q}(\gr B(A)(p))\Longrightarrow H^{p+q}(B(A))=\Tor^A_n(k,k).$$
Since $\gr B(A)\cong B(\gr A)$, the left-most term in the above
displayed equation is isomorphic to $\Tor^{\gr A}_{p+q}(k,k)(p)$, and
the assertion follows.
\end{proof}

The filtrations used in \cite[Theorem 3]{Ma} are bounded in one of the
directions; see \cite[Condition (4.1) or (4.1')]{Ma}. The filtration
in Proposition \ref{zzpro6.4} is not necessarily bounded in any
direction, but the algebra $A$ has an extra $\Lambda$-grading that
is $\Lambda$-locally finite and connected graded. This is the only
difference between \cite[Theorem 3]{Ma} and Proposition \ref{zzpro6.4}.
Now we return to connected ${\mathbb Z}$-graded algebras $A$.

\begin{lemma}
\label{zzlem6.5} Let $A$ be a filtered algebra and let $B=\gr_F A$.
Suppose that
\begin{enumerate}
\item[(a)]
$\Lambda={\mathbb Z}$,
\item[(b)]
$A$ is connected graded with respect to the $\Lambda$-grading,
\item[(c)]
$F$ is a filtration of $A$ and each $F(n)$
is a ${\mathbb Z}$-graded vector space, and
\item[(d)]
the Rees ring $\bigoplus_{n\in {\mathbb Z}} F(n) t^{-n}$ is noetherian.
\end{enumerate}
Then the following hold.
\begin{enumerate}
\item
$A$ and $B$ are noetherian.
\item
If $B$ has finite global dimension, so does $A$.
\item
If $B$ is a gci, then so is $A$.
\item
If $B$ is Gorenstein {\rm{(}}respectively, AS Gorenstein{\rm{)}}, so is $A$.
\item
If $B$ is AS regular, so is $A$.
\end{enumerate}
\end{lemma}

\begin{proof}
(1) This is  a graded version of
\cite[Proposition 1.2.3 in Chapter 2]{LvO}.

(2) This is  a graded version of
\cite[Theorem 7.2.11 in Chapter 1]{LvO}.

(3) By \eqref{E5.1.1}
$$gci(A)=\limsup_{n\to\infty}
\frac{\log (\sum_{i=0}^n \dim \Tor^{A}_i(k,k))}{\log n}.$$
Since $B$ is a complete intersection of GK-type, $gci(\gr A)<\infty$.
By Proposition \ref{zzpro6.4}, $gci(A)<gci(\gr A)<\infty$.
So $A$ is a gci.

(4) This is a consequence of \cite[Proposition 3.1]{Bj}.

(5) This follows from parts (4) and (2).
\end{proof}

For the rest of this section we assume the hypotheses of 
Lemma \ref{zzlem6.5}(a,b,c,d)
whenever a filtration appears. The following lemma is straightforward.

\begin{lemma}
\label{zzlem6.6}
Let $F$ be a filtration on a $\Lambda$-graded $A$, with the
degree being $\deg_{\Lambda}$. Then the associated graded algebra
$\gr_F A$ is a ${\mathbb Z}\times \Lambda$-graded algebra with
$\deg_{\gr} \widehat{a}=(n, \deg_{\Lambda} a)$ if $a\in F(n)\setminus F(n+1)$.
\end{lemma}

Let $g\in \Aut(A)$. We say $g$ {\it preserves the filtration} $F$ if $g(F(n))
\subset F(n)$ for all $n$. A subgroup $G\subset \Aut(A)$ {\it preserves
the filtration} $F$ if every $g\in G$ preserves $F$.

\begin{lemma}
\label{zzlem6.7}
Let $F$ be a filtration of $A$ with $B=\gr_F A$. Let $G$ be a subgroup
of $\Aut(A)$ such that $G$  preserves the filtration $F$.
\begin{enumerate}
\item
For each $g\in G$, there is an induced automorphism
$\widehat{g}: B\to B$ defined by
$$\widehat{g}(\widehat{a})=\widehat{g(a)}$$
for all $a\in F(n)\setminus F(n+1)$.
\item
The map $\Phi: g\to \widehat{g}$ defines a group homomorphism
from $G\to \Aut(B)$.
\item
If $G$ is finite, then $B^{\Phi(G)}=\gr_{F'} A^G$, where $F'$
is the restriction of $F$ to $A^G$.
\item
Every non-trivial element in the kernel of $\Phi$ has infinite order.
\end{enumerate}
\end{lemma}

\begin{proof} (1,2) This is clear.

(3) Since $F'(n)=F(n)\cap A^G$, we have $\gr_{F'} A^g\subset
B^{\Phi(G)}$ (even if $G$ is infinite). Now assume $G$ is finite,
and let $\int$ be the Reynolds operator $\frac{1}{|G|}\sum_{g\in G} g$.
Since $G$ is finite and ${\text{char }} k=0$, $\int$ defines an
exact functor. Starting with the short exact sequence
$$0\to F(n+1)\to F(n)\to F(n)/F(n+1)\to 0$$
for each $n$, we obtain that a short exact sequence
$$0\to \inth \cdot F(n+1)\to \inth\cdot F(n)\to \inth
\cdot (F(n)/F(n+1))\to 0,$$
which is
$$0\to F(n+1)^G\to F(n)^G\to (F(n)/F(n+1))^{\Phi(G)}\to 0.$$
Since $F(n)^G=F(n)\cap A^G$, the assertion follows.

(4) Let $g$ be an element in the kernel of $\Phi$. 
If $1\neq g\in G$ has finite order, we replace $G$ by the subgroup
generated by $g$. By part (3), $\gr_{F'} A^G=(\gr_F A)^{\Phi(G)}$.
Since $g\neq 1$, $A^G \neq A$. Thus $(\gr_F A)^{\Phi(G)}\neq \gr_F A$.
Thus $\Phi(G)\neq \{1\}$, or $\Phi(g)\neq 1$, and hence the assertion follows.
\end{proof}

In the next proposition, when working with the Hilbert series of
$\gr_F A$ and the trace of $g\in \Aut(\gr_F A)$, only the $\Lambda$-grading
(or the second grading) of $\gr_F A$ is used.

\begin{proposition}
\label{zzpro6.8}
Let $F$ be a filtration and $B=\gr_F A$. Consider $B$ as a connected
graded algebra using the second grading of $B$. Let $g\in \Aut(A)$ be such
that $g$ preserves the filtration $F$. Then the following
hold.
\begin{enumerate}
\item
$Tr_B(\Phi(g),t)=Tr_A(g,t)$.
\item
$H_B(t)=H_A(t)$. As a consequence, $B$ is cyclotomic if and only if
$A$ is.
\item
$g$ is a reflection if and only if $\Phi(g)$ is.
\item
$g$ is a bireflection if and only if $\Phi(g)$ is.
\item
$\hdet g=1$ if and only if $\hdet \Phi(g)=1$.
\item
Suppose both $A$ and $B$ are AS Cohen-Macaulay domains
\cite[Definition 0.1]{JoZ}. Then $B$ is AS Gorenstein
{\rm{(}}respectively, cyclotomic Gorenstein{\rm{)}} if and only if $A$ is.
\end{enumerate}
\end{proposition}

\begin{proof} (1) Again we consider the short exact sequence of graded
spaces
$$0\to F(n+1)\to F(n)\to F(n)/F(n+1)\to 0.$$
Then
$$Tr_A(g\mid_{F(n)}, t)-Tr_A(g\mid_{F(n+1)}, t)=
Tr_B(\Phi(g)\mid_{F(n)/F(n+1)},t).$$
Since $A$ is $\Lambda$-locally finite, so is $B$. Then, for each fixed
power $t^d$, the coefficients of $t^d$ in $Tr(\Phi(g)\mid_{F(n)/F(n+1)},t)$
are nonzero for only finitely many $n$. Thus
$$\begin{aligned}
Tr_B(\Phi(g),t)&=\sum_{n} Tr_B(\Phi(g)\mid_{F(n)/F(n+1)},t)\\
&=\sum_n Tr(g\mid_{F(n)}, t)-Tr(g\mid_{F(n+1)}, t)\\
&=Tr_A(g,t).
\end{aligned}$$

(2,3,4,5) These are consequence of (1).

(6) By Stanley's theorem \cite[Theorem 6.2]{JoZ}, the AS Gorenstein
property of $A$ is determined by the Hilbert series of $A$.
The assertions follow from part (2).
\end{proof}

Now we are ready to prove Proposition \ref{zzpro0.4}.

\begin{proof} (1) This follows from Lemmas \ref{zzlem6.5}(3) and
\ref{zzlem6.7}(3).

(2) This follows from Proposition \ref{zzpro6.8}(2,6) and Lemma
\ref{zzlem6.7}(3).

(3) This follows from Proposition \ref{zzpro6.8}(4).
\end{proof}

\begin{proposition}
\label{zzpro6.9} Let $A$ be a filtered algebra such that $B:=\gr_F
A$ is a noetherian connected graded AS Gorenstein domain.
Let $\Gamma$ be a subgroup of $\Aut(A)$, and let $\Gamma'=\Phi(\Gamma)$.
\begin{enumerate}
\item
If $B$ is $\Gamma'$-Kac, then $A$ is $\Gamma$-Kac.
\item
If $B$ is strongly $\Gamma'$-Kac, then $A$ is strongly $\Gamma$-Kac.
\end{enumerate}
\end{proposition}

\begin{proof} Let $G$ be a finite subgroup of $\Gamma$, and let
$G'=\Phi(G)$. The filtration $F$ on $A$ induces a filtration $F'$
on $A^G$ such that $\gr_{F'} A^G=B^{G'}$. The hypotheses
in Lemma \ref{zzlem6.5}(a-d) hold for the new filtration $F'$.

(1) Suppose now that $B$ is $\Gamma'$-Kac. Assume that Definition
\ref{zzdef5.7}(iii) holds for $(A,G)$. Then $A^G$ is cyclotomic
Gorenstein, and $G$ is generated by bireflections.
By Proposition \ref{zzpro6.8}(4) for $g\in G$, and
Proposition \ref{zzpro6.8}(6) for $A^G$, Definition
\ref{zzdef5.7}(iii) holds for $(B,G')$. Since
$B$ is $\Gamma'$-Kac, $B^{G'}$ is a complete intersection of GK-type.
By Lemma \ref{zzlem6.5}(3), $A^G$ is a  complete intersection of GK-type, and
the assertion follows.

(2) The proof is similar to the proof of part (1).
\end{proof}

\begin{definition}
\label{zzdef6.10} Let $A$ be an AS regular algebra and $\Gamma$ be a
subgroup of $\Aut(A)$. We say that $A$ is {\it
$\Gamma$-Shephard-Todd-Chevalley} (or {\it $\Gamma$-STC} for short) if,
for every finite subgroup $G\subset \Gamma$, the fixed subring $A^G$
is AS regular if and only if $G$ is generated by reflections
of $A$.
\end{definition}

\begin{proposition}
\label{zzpro6.10} Let $A$ be a filtered algebra such that $B:=\gr_F
A$ is a noetherian connected graded AS regular algebra.
Let $\Gamma$ be a subgroup of $\Aut(A)$, and let $\Gamma'=\Phi(\Gamma)$.
If $B$ is $\Gamma'$-STC, then $A$ is $\Gamma$-STC.
\end{proposition}

\begin{proof} Let $G$ be a finite subgroup of $\Gamma$ such that $G$
is generated by reflections. Then $G'\subset \Gamma'$ is generated by
reflections of $B$. Since $B$ is $\Gamma'$-STC, $B^{G'}$ is AS regular, so by
Lemma \ref{zzlem6.5}(5), $A^G$ is AS regular, proving one direction 
of the proposition.
The other direction follows from \cite[Proposition 2.5(b)]{KKZ3}.
\end{proof}

\section{Down-up algebras}
\label{zzsec7}

Down-up algebras, denoted by $A(\alpha,\beta,\gamma)$ for 
$\alpha, \beta, \gamma$
parameters in $k$, were defined by
Benkart-Roby \cite{BR} and studied by many others.  
They are $k$-algebras generated by $u$ and $d$,
and subject to the relations
$$\begin{aligned}
d^2u&=\alpha dud+\beta ud^2+\gamma d,\\
du^2&=\alpha udu+\beta u^2d+\gamma u.
\end{aligned}
$$
The universal enveloping algebra of the 3-dimensional Heisenberg Lie
algebra is $A(2,-1,0)$, and there are other interesting special cases,
see \cite{BR, KK}. In this paper, we are interested only in the
graded case, namely, when $\gamma=0$ (and $\deg u=\deg d=1$). It is
well-known that  $A(\alpha,\beta,0)$ is noetherian if and
only if it is AS regular if and only if $\beta\neq 0$ \cite{KMP}.
See \cite{BW} for some discussion about the representation theory
of graded down-up algebras. In this paper $A(\alpha,\beta)$
stands for $A(\alpha,\beta,0)$.

Throughout the rest of the paper, 
let $A=A(\alpha,\beta)$ be a noetherian graded down-up algebras.  We
let $a$ and $b$ be the roots of the ``character polynomial''
$$x^2-\alpha x-\beta=0,$$ and let 
$$\Omega_1 = du - aud \text{ and } \Omega_2 = du - bud.$$
Note that both $a,b \neq 0$ and $\Omega_1 = \Omega_2$ if and only if 
$\alpha^2 = -4\beta$.

\begin{lemma}
\label{zzlem7.1} Both $\Omega_1$ and $\Omega_2$ are regular normal
elements in $A$, and the following relations hold in $A$:
$$\begin{aligned}
u \Omega_1& =b^{-1} \Omega_1 u,\\
d \Omega_1& =b \Omega_1 d,\\
u \Omega_2& =a^{-1} \Omega_2 u,\\
d \Omega_2& =a \Omega_2 d,\\
ud \Omega_i&=\Omega_i ud \text{ for } i=1,2,\\
du \Omega_i&=\Omega_i du \text{ for } i=1,2,\\
\Omega_1\Omega_2&=\Omega_2\Omega_1.
\end{aligned}
$$
Furthermore, the first two equations 
{\rm{(}}as well as the next two equations{\rm{)}} are equivalent to the two
defining relations of $A$.
\end{lemma}

\begin{proof} It is easy to check that the first two 
(as well as the next two) are equivalent to
the two defining relations of $A$, and hence $\Omega_1$ and $\Omega_2$ are
normal elements of $A$. Since $A$ is a domain, these are regular elements.
The last three equations follow from the first four.
\end{proof}
Next we review the graded automorphism groups of $A$, and describe two
filtrations of $A$, and their associated graded rings, that will be used later.
\begin{lemma}
\label{zzlem7.2} Let $A=A(\alpha,\beta)$ with $\beta\neq 0$.
\begin{enumerate}
\item  \cite[Proposition 1.1]{KK} The graded automorphism group
of $A$ is given by
$$\Aut(A)=\begin{cases}
GL_2(k)& {\text{if}} \;(\alpha,\beta)=(0,1),\\
GL_2(k)& {\text{if}} \;(\alpha,\beta)=
(2,-1),\\
U=\biggl\{ \begin{pmatrix} a&0\\0&b\end{pmatrix}, \begin{pmatrix}
0&c\\d&0\end{pmatrix}: a,b,c,d\in k^\times\biggr\} &{\text{if}}
\; \beta=-1, \alpha\neq 2,\\
O=\biggl\{\begin{pmatrix} a&0\\0&b\end{pmatrix}: a,b\in
k^\times\biggr\} & {\text{otherwise.}}\end{cases}.$$
\item
Let $\Omega=du-aud$ with $\deg \Omega=2$. Define a filtration $F$ of $A$ by
$$\begin{aligned}
F(0)&=k,\\
F(-1)&=k+ kd+ku+k\Omega=:V,\\
F(-n)&=V^{n}, \; \forall \; n>1.
\end{aligned}
$$
Then $\gr_F A$ is isomorphic to the skew polynomial ring
$$B:=k\langle d,u,\Omega\rangle/
(d\Omega =b\Omega d, \Omega u=bu\Omega, du=aud).$$
\item
Suppose $\beta=-1$ and $\alpha\neq \pm 2$. Let $\Omega_1=du-aud$ and
$\Omega_2=du-b ud$ with $\deg \Omega_1=\deg \Omega_2=2$. Define a
filtration $F$ by
$$\begin{aligned}
F(0)&=k,\\
F(-1)&=k+ kd+ku+k\Omega_1+k \Omega_2=:V,\\
F(-n)&=V^{n}, \; \forall \; n>1.
\end{aligned}
$$
Then $\gr_F A$ is isomorphic to the algebra
$$B:=(k[d,u]/(du))\langle \Omega_1,\Omega_2\rangle
/(relations)$$
where $relations$ are
$$d\Omega_1 =b\Omega_1 d,\; \Omega_1 u=bu\Omega_1,\; d\Omega_2
=a\Omega_2 d, \; \Omega_2 u=au\Omega_2,\; [\Omega_1,\Omega_2]=0.$$
\end{enumerate}
\end{lemma}

\begin{proof} The proofs of (2) and (3) are similar, so we give 
only the proof of (3).

The elements of $\gr_F(A)$ satisfy the relations of $B$ so there is an 
algebra map from $\gr_F(A)$ onto $B$.  We can filter elements of $A$ by 
ordered pairs $\mathbb{N} \times \mathbb{N}$, where the first degree is 
the degree in the filtration $F$ and the second degree is the natural 
degree in the graded ring $A$ (so that $\deg(u) = \deg(d)= (1,1)$ and 
$\deg(\Omega_i) = (1,2)$). Considering the grading in second component,  
$\gr_F(A)$ is a graded ring with the same Hilbert series as
$A$ under the natural grading ($\deg(u) = \deg(d) = 1$).
Under this grading $B$ is a graded ring with Hilbert series
$$\frac{1-t^2}{(1-t^2)^2(1-t)^2,}$$
which is the same as the Hilbert series of $\gr_F(A)$. 
Hence the map from $\gr_F(A)$ to $B$ is an isomorphism.
\end{proof}

\begin{lemma}
\label{zzlem7.3} Let $A=A(\alpha,\beta)$ with $\beta\neq 0$
and $g$ be an element in $\Aut(A)\subset GL_2(k)$.
\begin{enumerate}
\item \cite[Theorem 1.5]{KK}
If $\lambda$ and $\mu$ are eigenvalues of $g$ as a matrix in $GL_2(k)$,
then the trace of $g$ is
$$Tr(g,t)=\frac{1}{(1-\lambda t)(1-\mu t)(1-\lambda \mu t^2)}.$$
\item \cite[Theorem 1.5]{KK}
The homological determinant of $g$ is
$$\hdet g=\lambda^2\mu^2=(\det g)^2.$$
As a consequence, if $H$ is a finite subgroup of $\Aut(A)$ with
trivial homological determinant, then either $H$ is a finite subgroup
of $SL_2(k)$, or $H$ satisfies the short exact sequence
\eqref{E3.0.1}.
\item
$g$ is a bireflection of $A$ if and only if either $\det
g=1$ {\rm{(}}i.e., $\lambda\mu=1${\rm{)}} or 
$g\in GL_2(k)$ is a classical reflection
{\rm{(}}i.e., either $\lambda=1$ or $\mu=1${\rm{)}}.
\item
Let $B=A(\alpha',\beta')$ be another down-up algebra. If $G\subset
\Aut(A)$ and $G\subset \Aut(B)$, where both $\Aut(A)$ and $\Aut(B)$ are
viewed as subgroups of $GL_2(k)$. Then $A^G$ is AS Gorenstein 
{\rm{(}}respectively, cyclotomic Gorenstein{\rm{)}} if and only if $B^G$ is.
\end{enumerate}
\end{lemma}

\begin{proof} (3) This follows from the definition and part (1).

(4) By Stanley's Theorem \cite[Theorem 6.1]{JoZ}, the property that
$A^G$ is AS Gorenstein is dependent only on the Hilbert series of
$A^G$. By Molien's Theorem, the Hilbert series of $A^G$ is 
dependent only on the trace $Tr(g,t)$ for all $g\in G$. By part (1),
$Tr(g,t)$ is dependent only on the eigenvalues of $g$. Hence the AS
Gorenstein property of $A^G$ is dependent only on the matrix
properties of $G$. Therefore $A^G$ is AS Gorenstein if and only if $B^G$
is.

By definition, the cyclotomic property is dependent only on the Hilbert
series of $A^G$, which is again dependent only on the matrix
properties of $G$. Therefore $A^G$ is cyclotomic Gorenstein if and
only if $B^G$ is.
\end{proof}

See also Proposition \ref{zzpro0.2} for further properties of
$A$. In particular, $A^H$ is AS Gorenstein if and only if $H$
has trivial homological determinant. To prove Theorem \ref{zzthm0.3}
we need analyze all finite subgroups $H\subset \Aut(A)$ with $\hdet=1$.
We start with some examples in the next section.



\section{AS Gorenstein fixed subrings for down-up algebras}
\label{zzsec8}

In this section, for some values of the parameters $\alpha$ and $\beta$, 
we discuss the fixed subrings $A^G$, where $A$ is a
noetherian graded down-up algebra, and $G$ is one of the groups $Q_i$, 
$i=1,\cdots,8$, that were introduced in Section \ref{zzsec3}.
All these  invariant subrings $A^G$ are AS Gorenstein by
Lemma \ref{zzlem7.3}(2) and Proposition \ref{zzpro0.2}(2), and our interest
is in whether the $A^G$ are complete intersections of some type.  In some
cases we are able to  express the algebra explicitly as a
classical complete intersection.  The down-up algebras
$A(-1,2)$ and $A(0,1)$, whose graded automorphism groups are
all of $GL_2(k)$, will be discussed in Section \ref{zzsec9}.

Let $A$ be a noetherian graded down-up algebra $A(\alpha,\beta)$
where $\beta\neq 0$. Let $a$ and $b$ be the roots of ``character
polynomial''
$$x^2-\alpha x-\beta=0.$$
Throughout most of this section we assume that $a\neq b$ 
(i.e. $\alpha^2 \neq -4\beta$); at the end of the section we consider 
$A(-2,-1)^{Q_i}$ for $i = 5, \ldots, 8$.
 Recall the normal elements 
$\Omega_1:=du-aud$ and $\Omega_2:=du-bud$, which are distinct when $a\neq b$.  
Below is a table that summarizes the results of this section.

$$\begin{array}{|c|c|c|c|c|}
\hline
 & & & &\\
\text{ Group } Q_i & A^{Q_i} & \text{ Generated by }& 
\text{ Cyclotomic ?}& \text{ Reference }\\
& &\text{ Bireflections? } & &\\
\hline
 Q_1 = \langle c_\epsilon \rangle & \text{hypersurface} &  
\text{ yes } &  \text{ yes } & \text{Example }\ref{zzex8.2}\\
\hline
Q_2 = \langle d_1, c_\epsilon \rangle & \text{ cci } &  
\text{ yes } & \text{ yes }& \text{Example }\ref{zzex8.3}\\
n \text{ even } & & & &\\
\hline
Q_3 = \langle d_1, c_\epsilon \rangle & n=1 \text{ cci } & 
\text{ yes } & \text{ yes }&\text{Lemma }\ref{zzlem8.4}\\
n \text{ odd }& n \geq 3 \text{ no type of ci } & 
\text{ yes }& \text{ no }&\\
\hline
Q_4 = \langle d_1, c_{\epsilon, -} \rangle & 
\text{ no type of ci } & \text{ no }& \text{ no } & 
\text{Lemma }\ref{zzlem8.7}\\
\hline
Q_5 = \langle s_1, c_\epsilon \rangle & \text{ gci } & 
\text{ yes } & \text{ yes }& \text{Proposition }\ref{zzpro8.8}\\
n \text{ even} & & & &\\
\hline
Q_6 = \langle s, c_\epsilon \rangle & \text{ gci } & 
\text{ yes } & \text{ yes } & \text{Proposition }\ref{zzpro8.8}\\
\hline
Q_7 = \langle d_1, s, c_\epsilon \rangle & \text{ gci } & 
\text{ yes } & \text{ yes } & \text{Proposition }\ref{zzpro8.8}\\
n \text{ even} & & & & \\
\hline
Q_8 = \langle s, c_{\epsilon,-} \rangle & \text{ gci } & 
\text{ yes } & \text{ yes } & \text{Proposition }\ref{zzpro8.8}\\
n = 4k & & & &\\
\hline
\end{array}$$

\begin{center}
Table 4: Invariants of down-up algebras with $a \neq b$ under 
groups $Q_i, i = 1, \ldots, 8$,
and $A(-2,-1)$ for $Q_i, i = 5, \ldots, 8$.
\end{center}

\medskip

The following lemma follows by a straightforward induction.

\begin{lemma}
\label{zzlem8.1} Let $\Omega_1$ and $\Omega_2$ be defined as above.
Let $n$ be a positive integer. Then
$$\begin{aligned}
ud&=\frac{1}{b-a} \; (\Omega_1-\Omega_2),\\
du&=\frac{1}{b-a} \; (b\Omega_1-a\Omega_2),\\
u^n d^n&=f_{n}(\Omega_1,\Omega_2),\\
d^n u^n&=g_n(\Omega_1,\Omega_2),
\end{aligned}
$$
where
$$f_{n}(X,Y)=\frac{\prod_{i=1}^{n} (b^{-(i-1)}X-a^{-(i-1)}Y)}
{(b-a)^n}, \quad g_n(X,Y)=\frac{\prod_{i=1}^{n} (b^{i}X-a^{i}Y)}
{(b-a)^n}.$$
Furthermore, $f_n(b^n X, a^n Y)=g_n(X,Y)$.
\end{lemma}

Let $Q_1\subset O$ be the cyclic group $C_n$ generated
$c_{\epsilon}$, where $\epsilon$ is a primitive $n$th root of unity,
see Lemma \ref{zzlem3.4}(1).

\begin{example}
\label{zzex8.2} The invariant subring $A^{Q_1}$ is generated by
$$X_1:= u^n, X_2:=d^n, X_3:=\Omega_1, \; {\text{and}}\;
X_4:=\Omega_2$$ subject to the relations
$$\begin{aligned}
r_1: \quad X_4 X_3&=X_3X_4,\\
r_2: \quad X_4 X_2&=a^{-n} X_2 X_4,\\
r_3: \quad X_4 X_1&=a^n X_1 X_4,\\
r_4: \quad X_3 X_2&=b^{-n} X_2 X_3,\\
r_5: \quad X_3 X_1&=b^n X_1 X_3,\\
r_6: \quad X_2 X_1&=X_1X_2+g_n(X_3,X_4)-f_n(X_3,X_4),\\
r_7: \quad X_1 X_2&=f_n(X_3,X_4).
\end{aligned}
$$
The first six relations define an iterated Ore extension that is a
connected graded, AS regular algebra $C_1:=k\langle
X_1,\cdots,X_4\rangle/(r_1,\cdots,r_6)$. The element
$$\Theta:=X_1 X_2-f_n(X_3,X_4)$$
is a regular central element in $C_1$ and $A^{Q_1}\cong C_1/(\Theta)$,
and hence $A^{Q_1}$ is a hypersurface.
\end{example}

\begin{proof}  
Using the fact that 
$f_n(b^n X, a^n Y)=g_n(X,Y)$, it follows that $\Theta$ is central.
By relation $r_7$ there is a graded surjection from $C_1/(\Theta) 
\to A^{Q_1}$.  We show
this map is a graded isomorphism by showing these algebras have the 
same Hilbert series.

The Hilbert series of $C_1/(\Theta)$ is
$$H_{C_1/(\Theta)}(t) = \frac{1-t^{2n}}{(1-t^n)^2(1-t^2)^2},$$
while using Molien's Theorem and Lemma \ref{zzlem7.3}(1), 
it follows that the Hilbert series of
$A^{Q_1}$ is
$$ \frac{1}{n} \sum_{i=0}^{n-1} \frac{1}{(1-\epsilon^i t)
(1-\epsilon^{-i} t)(1-t^2)} = h(t) \frac{1}{(1-t^2)}.$$
But $h(t)$ is the Hilbert series of $k[t_1,t_2]^{Q_1}$, and it 
follows from the description of ($A_n$) that precedes
Lemma \ref{zzlem3.1} that
$$H_{k[t_1,t_2]^{Q_1}}(t) = \frac{1-t^{2n}}{(1-t^n)^2(1-t^2)}.$$
Hence $H_{A^{Q_1}}(t)=H_{C_1/(\Theta)}(t)$. As a consequence, 
$A^{Q_1}\cong C_1/(\Theta)$.
\end{proof}

Let $Q_2\subset O$ be the group of order $4n$, which is generated 
by $d_1$ and $c_{\epsilon}$ for a primitive
$2n$th root of unity $\epsilon$, that occurred in Lemma \ref{zzlem3.4}(2).  
Let $\tilde{f}_{2n}(x_1,x_2,x_3)$ be a (commutative) polynomial in
$x_1,x_2,x_3$ such that
$$f_{2n}(X,Y)=\tilde{f}_{2n}(X^2,XY,Y^2).$$ Let
$\tilde{g}_{2n}(x_1,x_2,x_3)$ be the polynomial defined by
$$\tilde{g}_{2n}(x_1,x_2,x_3)=\tilde{f}_{2n}(b^{4n}x_1, a^{2n}b^{2n}
x_2, a^{4n}x_3).$$
Then $g_{2n}(X,Y)= \tilde{g}_{2n}(X^2,XY,Y^2)$.

\begin{example}
\label{zzex8.3} The invariant subring $A^{Q_2}$
is generated by
$$X_1:= u^{2n}, X_2:=d^{2n}, X_3:=\Omega_1^2,\;
X_4:=\Omega_1\Omega_2,\; {\text{and}}\;
X_5:=\Omega_2^2,$$
and subject to the relations
$$\begin{aligned}
r_1: \quad X_5X_4&=X_4X_5,\\
r_2: \quad X_5X_3&=X_3X_5,\\
r_3: \quad X_5X_2&=a^{-4n} X_2X_5,\\
r_4: \quad X_5X_1&=a^{4n} X_1X_5,\\
r_5: \quad X_4X_3&=X_3X_4,\\
r_6: \quad X_4X_2&=a^{-2n}b^{-2n} X_2X_4,\\
r_7: \quad X_4X_1&=a^{2n}b^{2n}X_1X_4,\\
r_8: \quad X_3X_2&=b^{-4n}X_2X_3,\\
r_9: \quad X_3X_1&=b^{4n} X_1X_3,\\
r_{10}:\quad X_2X_1&=X_1X_2+\tilde{g}_{2n}(X_3,X_4,X_5)-
\tilde{f}_{2n}(X_3,X_4,X_5),\\
r_{11}: \quad X_1X_2&= \tilde{f}_{2n}(X_3,X_4,X_5),\\
r_{12}: \quad X_3X_5&=X_4^2.
\end{aligned}
$$
The first ten relations define a connected graded,
AS regular algebra
$$C_2:=k\langle X_1,\cdots,X_5\rangle/(r_1,\cdots,r_{10})$$ 
that is an iterated Ore extension. The elements
$$\Theta_1:=X_1 X_2-\tilde{f}_{2n}(X_3,X_4,X_5)\quad \text{ and }
\Theta_2:= X_4^2-X_3X_5$$ form a sequence of regular normalizing elements
in $C_2$ and $A^{Q_2}\cong C_2/(\Theta_1,\Theta_2)$. As a
consequence, $A^{Q_2}$ is a cci.
\end{example}

\begin{proof} 
It is not hard to check that $\Theta_2$ is normal in $C_2$.  The element
$\Theta_1$ is central in $C_2$ because 
$\tilde{g}_{2n}(x_1,x_2,x_3)=\tilde{f}_{2n}(b^{4n}x_1, a^{2n}b^{2n}
x_2, a^{4n}x_3)$ implies 
$$X_i \tilde{g}_{2n}(x_1,x_2,x_3)
=  \tilde{f}_{2n}(x_1, x_2, x_3)X_i$$ 
for i=1,2,
since both sides of the equation have the same monomial summands; 
furthermore, $X_3, X_4$ and $X_5$ clearly commute with $\Theta_1$,
so that $\Theta_1$ is central in $C_2$.  Clearly $\Theta_2$ is a 
nonzero element of a domain, and
$\Theta_1$ is regular in $C_2/(\Theta_2)$ because its leading term 
does not contain $X_3,X_4$ or $X_5$.
Hence $\Theta_2, \Theta_1$ is a regular normalizing sequence.

There is a graded homomorphism from $B:=C_2/(\Theta_1,\Theta_2)$ 
onto $A^{Q_2}$, since the elements
$X_i$ satisfy the relations in $B$.  We show that this map is 
injective by showing
that $B$ and $A^{Q_2}$ have the same Hilbert series.  The Hilbert 
series of $B$ is
$$H_{B}(t) = \frac{(1-t^{4n})(1-t^8)}{(1-t^{2n})^2(1-t^4)^3} = 
\frac{(1-t^{4n})(1+t^4)}{(1-t^{2n})^2(1-t^4)^2}.$$
By Molien's Theorem and Lemma \ref{zzlem7.3}(1) the Hilbert series 
of $A^{Q_2}$ is
$$\begin{aligned}
H_{A^{Q_2}}(t) &= \frac{1}{4n} 
\biggl(\sum_{i=0}^{2n-1}\frac{1}{(1-\epsilon^i t)(1-\epsilon^{-i} t)(1-t^2)} 
+\sum_{i=0}^{2n-1}\frac{1}{(1+\epsilon^i t)(1-\epsilon^{-i} t)(1+t^2)}
\biggr)\\
&:= \frac{1}{4n}(S_1 \frac{1}{(1-t^2)} + S_2 \frac{1}{(1+t^2)}),
\end{aligned}
$$
where $S_1$ and $S_2$ are the indicated sums.  We can compute $S_1$ and 
$S_2$ from the commutative case, using the fact that $k[t_1,t_2]^{Q_2}$ 
is generated by $t_1^{2n}, t_{2}^{2n}, t_1^2 t_2^2$, which
is the hypersurface $k[Z_1,Z_2,Z_3]/(Z_1Z_2-{Z_3}^n)$, so has Hilbert series
$$H_{k[t_1,t_2]^{Q_2}}(t) = \frac{1-t^{4n}}{(1-t^{2n})^2(1-t^4)},$$
which by Molien's Theorem is
$$H_{k[t_1,t_2]^{Q_2}}(t) = \frac{1}{4n}(S_1 + S_2).$$
We can compute $S_1$ from the Hilbert series of 
$k[t_1,t_2]^{(c_\epsilon)}$, where $(c_\epsilon)$ is the cyclic group
of order $2n$, and Molien's
Theorem, and we obtain
$$S_1 := \sum_{i=0}^{2n-1}\frac{1}{(1-\epsilon^i t)(1-\epsilon^{-i} t)} 
= \frac{2n(1-t^{4n})}{(1-t^{2n})^2(1-t^2)}.$$
We compute $S_2$ using the Hilbert series of $k[t_1, t_2]^{Q_2}$, 
Molien's Theorem, and $S_1$, and we obtain
$$\begin{aligned}
S_2&:= \sum_{i=0}^{2n-1}\frac{1}{(1+\epsilon^i t)(1-\epsilon^{-i} t)}
= 4n H_{k[t_1,t_2]^{Q_2}}(t) -S_1\\
&= \frac{2n(1-t^{4n})(1-t^2)}{(1-t^{2n})^2(1-t^4)}.
\end{aligned}
$$
Then 
$$\begin{aligned}
H_{A^{Q_2}}(t) 
&= \frac{1}{4n}(S_1 \frac{1}{(1-t^2)} + S_2 \frac{1}{(1+t^2)})\\
&= \frac{1}{4n}\biggl(\frac{2n(1-t^{4n})}{(1-t^{2n})^2(1-t^2)} 
\frac{1}{(1-t^2)} + \frac{2n(1-t^{4n})(1-t^2)}{(1-t^{2n})^2(1-t^4)} 
\frac{1}{(1+t^2)}\biggr)\\
&= \frac{(1-t^{4n})(1+t^4)}{(1-t^{2n})^2(1-t^4)^2},
\end{aligned}$$
completing the proof.
\end{proof}

We next consider the case where we have an odd root of unity.

\begin{lemma}
\label{zzlem8.4}
Let $Q_3\subset O$ be the group occurring in Lemma 
{\rm{\ref{zzlem3.4}(3)}}, 
i.e. the group generated by $d_1$ and
$c_{\epsilon}$, for $\epsilon$ a primitive $n$th root of unity, 
for an odd integer $n$.
\begin{enumerate}
\item
The group $Q_3$ is a bireflection group for $A$.
\item
The invariant subring $A^{Q_3}$ is AS Gorenstein.
\item
$A^{Q_3}$ is a cci if $n=1$.
\item
$A^{Q_3}$ is not cyclotomic if $n>1$.
\end{enumerate}
\end{lemma}

\begin{proof} (1) Both $d_1$ and $c_{\epsilon}$ are
bireflections of $A$ by Lemma \ref{zzlem7.3}(3), and so $Q_3$ 
is a bireflection group for $A$.

(2) This follows from Proposition \ref{zzpro0.2}(2), as $(\det g)^2=1$ for all
$g\in Q_3$ and hence $\hdet g=1$ by Lemma \ref{zzlem7.3}(2).

(3) When $n=1$, $Q_3$ is the cyclic group of order 2 generated by $d_1$, and
$A^{Q_3}$ is generated by
$$X_1:=d, \; X_2=u^2, \; X_3:=u(du-a ud), \; {\text{and}}\; X_4:=(du-aud)^2$$
subject to the relations
$$\begin{aligned}
r_1: \quad X_1 X_2&= a^2 X_2X_1+(a+b)X_3,\\
r_2: \quad X_1 X_3&=ab X_3X_1+X_4,\\
r_3: \quad X_1 X_4&= b^2 X_4X_1,\\
r_4: \quad X_2 X_3&=b^{-2} X_3X_2,\\
r_5: \quad X_2 X_4&=b^{-4} X_4 X_2,\\
r_6: \quad X_3 X_4&= b^{-2} X_4X_3,\\
r_7: \quad X_2X_4 &=b^{-1} X_3^2.
\end{aligned}
$$
The first six relations define an AS regular algebra $C$.
Let  $\Omega:=X_2X_4-b^{-1} X_3^2$; it is not hard to check that 
$\Omega$ is normal in $C$,  and that
there is a graded homomorphism from $C/(\Omega)$ onto $A^{Q_3}$.  By 
Molien's Theorem,
the Hilbert series of $A^{Q_3}$ is
$$ \frac{1}{2} \biggl( \frac{1}{(1-t)^2 (1-t^2)} + 
\frac{1}{(1-t)(1+t)(1+t^2)}\biggr) =  
\frac{1-t^6}{(1-t)(1-t^2)(1-t^3)(1-t^4)},$$
which is the Hilbert series of $C/(\Omega)$.  Hence $A^{Q_3}$
is isomorphic to $C/(\Omega)$, and so $A^{Q_3}$
is a cci.

(4) Assume $n>1$. Since
$Q_3\subset O$, $Q_3$ preserves the filtration $F$ defined in Lemma
\ref{zzlem7.2}(2). By Proposition \ref{zzpro6.8} we need to 
show only that $(\gr_F A)^{\Phi(Q_3)}$ is not cyclotomic. By Lemma
\ref{zzlem7.2}(2), $\gr_F A$ is isomorphic to skew polynomial ring
$$B:=k\langle t_1, t_2, t_3\rangle/(t_2t_3=b t_3 t_2, b t_1t_3=t_3 t_1,
t_2 t_1=at_1 t_2),$$ where $t_1:=\widehat{u}, t_2:=\widehat{d}$ and
$t_3:=\widehat{\Omega}$. Hence $\deg t_1=\deg t_2=1, \deg t_3=2$, and
the group $\Phi(Q_3)\subset \Aut(B)$ is generated by $\widehat{d_1}:
t_1\to -t_1, t_2\to t_2, t_3\to -t_3$ and $\widehat{c_{\epsilon}}:
t_1\to \epsilon t_1, t_2\to \epsilon^{-1}t_2, t_3\to t_3$. Let $G$ be
the subgroup of $\Phi(Q_3)$ generated by $\widehat{c_{\epsilon}}$.
Then $B^G$ is isomorphic to the algebra
$S:=(k_{p_{ij}}[X_1,X_2,X_3]/(f))[t_3;\sigma]$, where $X_1:=t_1^n,
X_2:=t_2^n, X_3:=t_1t_2$ and $f:=X_2X_1-a^{{n \choose 2}}X_3^n$ 
(for some appropriate $p_{ij}$'s and $\sigma$, which are not essential
for the argument). The automorphism $\widehat{d_1}$ induces an
automorphism $h\in \Aut(S)$, which is determined by
$$h: X_1\to -X_1, \; X_2\to X_2, \; X_3\to -X_3, \; t_3\to -t_3.$$
By Lemma \ref{zzlem1.6}(3), the trace of $h$ is
$$Tr_S(h, t)=\frac{1+t^{2n}}{(1+t^n)(1-t^n)(1+t^2)(1+t^2)},$$
and the Hilbert series of $S$ is
$$H_S(t)=Tr_S(Id_S, t)=\frac{1-t^{2n}}{(1-t^n)(1-t^n)(1-t^2)(1-t^2)}.$$
Since $h$ is an involution, by Molien's theorem,
$$\begin{aligned}
H_{A^{Q_3}}(t)&=H_{B^{\Phi(Q_3)}}(t)=H_{S^h}(t)\\
&=\frac{1}{2}\biggl(Tr_S(h, t)+H_S(t)\biggr)\\
&=\frac{(1+t^n+t^{2n})(1+t^4)+2t^{n+2}}{(1-t^{2n})(1-t^4)^2}.
\end{aligned}
$$
By the following lemma, Lemma \ref{zzlem8.6}(a),
$(1+t^n+t^{2n})(1+t^4)+2t^{n+2}$ is not a product of cyclotomic
polynomials. Therefore $A^{Q_3}$ is not cyclotomic.
\end{proof}

\begin{remark}
\label{zzrem8.5}  In the classical case there
are groups generated by bireflections, where the fixed ring is not
a complete intersection.
Let $G$ be the finite subgroup of $SL_3(k)$ generated by
$$\begin{pmatrix} \epsilon &0&0\\0&\epsilon^{-1}&0\\0&0&1\end{pmatrix}
\quad {\text{and}}\quad
\begin{pmatrix} -1 &0&0\\0&1&0\\0&0&-1\end{pmatrix},$$
where $n$ is a primitive $n$th root of unity for odd integer
$n\geq 3$. Both matrices are
classical bireflections, so that $G$ is a bireflection group for
$k[t_1,t_2,t_3]$. A Hilbert series argument similar to the one
in the proof of the previous lemma  shows that $k[t_1,t_2,t_3]^G$ is
not cyclotomic, whence, not a complete intersection.
(In this case the numerator is $f(t)= t^{2n} + t^{n+2} + t^{n} + 1$ and 
$t= -1$ is a root;  further $f'(-1)=2$, so
$f(t)$ is increasing at $t=-1$.  Since 
$\lim_{t\to -\infty} f(t) = \infty$, there must be another real root smaller
than $-1$ (See also \cite{WR, Wa}).
\end{remark}

\begin{lemma}
\label{zzlem8.6} Each of the following polynomials is not a product of
cyclotomic polynomials.
\begin{enumerate}
\item
$q(t)=(1+t^n+t^{2n})(1+t^4)+2t^{n+2}$ for any $n>1$.
\item
$q(t)=(1+t^{2n})(1+t^4)+4t^{n+2}$ for any $n\geq 1$.
\end{enumerate}
\end{lemma}

\begin{proof} (1)
Set $t=e(s) := e^{2\pi i s}$.  Then $f_n(t)$ is cyclotomic if and
only if $f_n(e(s))$ has $2n+4$ real roots in $[0,1)$.  Let
$$\begin{aligned}
g_n(s) &= \frac{e(-(n+2)s)}{4} f_n(e(s))\\
&=\cos(4 \pi s) \cos(2 \pi ns) + \cos^2(2 \pi s)
\end{aligned}
$$
(using the addition formula for $\cos(\alpha + \beta)$). Then for $0
\leq s \leq 1/8$ we have
$$\cos(4 \pi s) (\cos(2 \pi ns)+1) \geq 0,$$
and consequently
$$g_n(s) \geq \sin^2(2 \pi s) >0$$
for $0 < s \leq 1/8$.  But, for $n > 2$, we compute
$$g'(1/4n) = -2 \pi ( n \cos(\pi/n) + \sin(\pi/n)) < 0,$$
and
$$g'_n(1/2n) = 2 \pi \sin(2 \pi/n)> 0,$$
so $g_n(s)$ has a local minimum at a point $s_0 \in (1/4n,1/2n)$,
and $g_n(s_0) > 0$ if $n \geq 4$.  Thus $f_n(t)$ has a least one
root off the unit circle for $n \geq 4$, and one can check the cases
$n=2$ and $n=3$.

(2) As in the previous argument, let $t=e(s) := e^{2\pi i s}$.  Then
$f_n(t)$ is cyclotomic if and only if $f_n(e(s))$ has $2n+4$ real
roots in $[0,1)$.  Let
$$\begin{aligned}
g_n(s) &= \frac{e(-(n+2)s)}{4} f_n(e(s))\\
&=\cos(4 \pi s) \cos(2 \pi ns) + 1.
\end{aligned}$$
Hence $s$ is a root of $g_n(s)$ if and only if $s$ is a root of
$f_n(e(s))$, which happens in the following two cases:
either $\cos(4 \pi s) = 1$ and $\cos(2 \pi ns) = -1$, or
$\cos(4 \pi s) = -1$ and $\cos(2 \pi ns) = 1$.

\noindent
Case a: $\cos(4 \pi s) = 1$ implies that $s= 0$ or $1/2$ since $s
\in [0,1)$, and $\cos(2 \pi ns) = -1$ implies that
$s=1/2$ (and this equation holds only if $n$ is odd).  In this case,
the only possible cyclotomic root of $f_n(t)$ can be $e^{\pi i} = -1$.

\noindent
Case b: $\cos(2 \pi ns) = -1$ implies that $s= 1/4$ or $3/4$, and in
either case $\cos(2 \pi ns) = 1$ implies that $n$ is a multiple of
$4$, and the only cyclotomic roots of $f_n(t)$ are $\pm i$.

\noindent
Hence if $n$ is odd the only cyclotomic factor can be $t+1$, and for
any $n \geq 0$, $f_n(t) \neq k (t+1)^m$ for any $k \in k$.
When $n$ is even (and a multiple of 4) the only possible cyclotomic
factor is $t^2+1$,
 but for any $n >0$ $f_n(t) \neq k (t^2+1)^m$ for any $k \in k$.
\end{proof}

Similar to Lemma \ref{zzlem8.4} we have the following.

\begin{lemma}
\label{zzlem8.7}
Let $Q_4\subset O$ be the group generated by  $c_{\epsilon, -}$, 
for a primitive
$4n$th root of unity $\epsilon$ for any $n\geq 1$, which occurred 
in Lemma \ref{zzlem3.4}(4).
The invariant subring $A^{Q_4}$ is AS Gorenstein, but
not cyclotomic. The group $Q_4$ is not a bireflection group.
\end{lemma}

\begin{proof} Every element in $Q_4$ is of the form $c_{\epsilon,
-}^i$ for $i=0,1,\cdots, 4n-1$. It is easy to see that the trace of
$c_{\epsilon, -}^i$ (as an automorphism of $A$) is
$$Tr_A(c_{\epsilon, -}^i, t)=\frac{1}{(1-(-1)^i\epsilon^i
t)(1-\epsilon^{-i}t)(1-(-1)^i t^2)}.$$ Hence $c_{\epsilon,-}^i$ is a
bireflection if and only if $i$ is even, and this implies that $Q_4$ is
not a bireflection group for $A$.

By Lemma \ref{zzlem1.3}, one sees that  $\hdet c_{\epsilon,-}^i=1$ for
all $i$. Therefore $A^{Q_4}$ is AS Gorenstein [Proposition \ref{zzpro0.2}(2)].
It remains to show that $A^{Q_4}$ is not cyclotomic.

Since $Q_4\subset O$, $Q_4$ preserves the filtration $F$ defined in
Lemma \ref{zzlem7.2}(2). By Proposition \ref{zzpro6.8}(2) we need
to show only that $(\gr_F A)^{\Phi(Q_4)}$ is not cyclotomic. By Lemma
\ref{zzlem7.2}(2), $\gr_FA$ is isomorphic to skew polynomial ring
$$B:=k\langle t_1, t_2, t_3\rangle/(t_2t_3=b t_3 t_2, b t_1t_3=t_3 t_1,
t_2 t_1=at_1 t_2)$$ where $t_1:=\widehat{u}, t_2:=\widehat{d}$ and
$t_3:=\widehat{\Omega}$. Hence $\deg t_1=\deg t_2=1, \deg t_3=2$ and
the group $\Phi(Q_4)\subset \Aut(B)$ is generated
$\phi:=\Phi(c_{\epsilon,-}):t_1\to -\epsilon t_1, t_2\to
\epsilon^{-1}t_2, t_3\to -t_3$, where $\epsilon$ is a primitive
$4n$th root of unity. Let $G$ be the subgroup of $\Phi(Q_4)$
generated by $\phi^2$. Then $B^G$ is isomorphic to the algebra
$S:=(k_{p_{ij}}[X_1,X_2,X_3]/(f))[t_3;\sigma]$, where $X_1:=t_1^{2n},
X_2:=t_2^{2n} X_3:=t_1t_2$, where $p_{12}:=a^{4n^2}$, $p_{13}:=a^{2n}$
and $p_{23}:=a^{-2n}$, and where $f:=X_1X_2-a^{-{2n \choose
2}}X_3^{2n}$. The automorphism $\phi\in \Aut(B)$ induces an
involution $h\in \Aut(S)$, which is determined by
$$h: X_1\to -X_1, \quad X_2\to -X_2, \quad X_3\to -X_3, \quad t_3\to -t_3.$$
Note that $h(f)=f$. By Lemma \ref{zzlem1.6}(3), the trace of $h$ is
$$Tr_S(h, t)=\frac{1-t^{4n}}{(1+t^{2n})(1+t^{2n})(1+t^2)(1+t^2)},$$
and the Hilbert series of $S$ is
$$H_S(t)=Tr_S(Id_S, t)=\frac{1-t^{4n}}{(1-t^{2n})(1-t^{2n})(1-t^2)(1-t^2)}.$$
Since $h$ is an involution, by Molien's theorem,
$$\begin{aligned}
H_{A^{Q_4}}(t)&=H_{B^{\Phi(Q_4)}}(t)=H_{S^h}(t)\\
&=\frac{1}{2}\biggl(Tr_S(h, t)+H_S(t)\biggr)\\
&=\frac{(1-t^{4n})
[(1+t^{4n})(1+t^4)+4t^{2n+2}]}{(1-t^{4n})^2(1-t^4)^2}.
\end{aligned}
$$
By the previous lemma, Lemma \ref{zzlem8.6}(b),
$(1+t^{4n})(1+t^4)+4t^{2n+2}$ is not a product of cyclotomic
polynomials. Therefore $A^{Q_4}$ is not cyclotomic.
\end{proof}

\begin{proposition}
\label{zzpro8.8} Suppose that $\beta=-1$ and $\alpha\neq 2$. Let $H$
be any of the groups $Q_5$, $Q_6$, $Q_7$ or $Q_8$ in Lemma \ref{zzlem3.4}. 
Then the invariant subring $A^{H}$ is a gci.
\end{proposition}

\begin{proof} First we assume that $\alpha\neq -2$.
Since $\beta=-1$ and $\alpha\neq \pm 2$, $b=a^{-1}\neq
\pm 1$, so that $\Omega_1=du-aud$ and $\Omega_2=du-a^{-1} ud$ are 
distinct elements of $A$.

We consider each of the four different groups $Q_5$, $Q_6$, $Q_7$, 
and $Q_8$ for $A=A(\alpha, -1)$
when $\alpha\neq -2$. Each of these groups preserves the filtration
of Lemma \ref{zzlem7.2}(3), defined by $F(n)=V^{-n}$, where $V=kd+ku+k\Omega_1
+k\Omega_2$, and the associated graded ring is isomorphic to
$$B:=(k[t_1,t_2]/(t_1t_2))\langle \Omega_1,\Omega_2\rangle
/(relations),$$
where $relations$ are
$$t_1\Omega_1 =a^{-1}\Omega_1 t_1,\; \Omega_1 t_2=a^{-1}t_2 \Omega_1,
\; t_1\Omega_2
=a\Omega_2 t_1, \; \Omega_2 t_2=at_2\Omega_2,\; [\Omega_1,\Omega_2]=0.$$
By Proposition \ref{zzpro0.4}(1), in each of the four cases it 
suffices to show that $B^{\Phi(H)}$
is a gci.\\
~\\
{\bf Case (i): $H=Q_5$.}\\
Recall that $Q_5\subset U$ is the binary dihedral group $BD_{4n}$
generated $s_1$ and $c_{\epsilon}$, where $\epsilon$ is a primitive
$2n$th root of unity. Here we also include the abelian case $BD_{4}$.
Let $G$ be the subgroup of $\widehat{H}:=\Phi(H)$ generated by
$\widehat{c_{\epsilon}}$ which sends
$$t_1\to \epsilon t_1, \quad t_2\to \epsilon^{-1} t_2, \quad
\Omega_1\to \Omega_1, \quad \Omega_2\to \Omega_2.$$
It is easy to see that, by setting $x_i=t_i^{2n}$ for $i=1,2$, $B^G$ is
isomorphic to
$$D:=(k[x_1,x_2]/(x_1x_2))\langle \Omega_1,\Omega_2\rangle
/(relations)$$
where $relations$ are
$$x_1\Omega_1 =a^{-2n}\Omega_1 x_1,\; \Omega_1 x_2=a^{-2n}x_2 \Omega_1,
\; x_1\Omega_2
=a^{2n}\Omega_2 x_1, \; \Omega_2 x_2=a^{2n}x_2\Omega_2,$$
and
$$[\Omega_1,\Omega_2]=0.$$
It is easy to check that $D\cong T_{a^{2n}}/(x_1x_2)$, where
$T_{a^{2n}}$ was studied in
Example \ref{zzex5.9}. By identifying $\Omega_i\in D$ with $y_i\in
T_{a^{2n}}$ we have $D=T_{a^{2n}}/(x_1x_2)$.
The quotient group $\widehat{H}/G$ is generated by the involution
$\widehat{s_1}$, which sends
$$x_1\to x_2, \quad x_2\to x_1, \quad
\Omega_1\to a\Omega_2, \quad \Omega_2\to a^{-1}\Omega_1.$$
In particular, $\widehat{s_1}$ preserves the relation $x_1x_2=0$.
Therefore
$$B^{\Phi(H)}=D^{\widehat{s_1}}=(T_{a^{2n}}/(x_1x_2))^{\widehat{s_1}}
\cong (T_{a^{2n}})^{\widehat{s_1}}/(x_1x_2),$$
where the last equation follows from Lemma \ref{zzlem5.4}(2). Therefore
it suffices to show that $(T_{a^{2n}})^{\widehat{s_1}}$ is a gci.
Now $\widehat{s_1}=\zeta_2$ as defined in Example \ref{zzex5.9} for $c=a$.
By Example \ref{zzex5.9}(ii), $(T_{a^{2n}})^{\widehat{s_1}}$ is a cci,
whence a gci by Lemma \ref{zzlem5.3}(1). This finishes the first case.\\
~\\
{\bf Case (ii): $H=Q_6$.}\\
The first part of argument is similar to the
argument for $H=Q_5$ (where we replace $2n$ by $n$). The
quotient group $\widehat{H}/G$ is generated by the involution
$\widehat{s}$, which sends
$$x_1\to x_2, \quad x_2\to x_1, \quad  \Omega_1\to -a\Omega_2, 
\quad \Omega_2\to -a^{-1}\Omega_1.$$
So the difference from the previous case 
is that $\widehat{s_1}=\zeta_2$ for $c=-a$ (not $a$).
Hence $A^{Q_6}$ is a gci.\\
~\\
{\bf Case (iii): $H=Q_7$.}\\
The first part of the proof is similar to the argument
above. Recall that $Q_7\subset U$ is generated $d_1$, $s$ and
$c_{\epsilon}$, where $\epsilon$ is a primitive $2n$th root of unity.
We use again the filtration of Lemma \ref{zzlem7.2}(3), and let $B$ be the
associate graded algebra. Let $G$ be the subgroup of $\widehat{H}:=\Phi(H)$ 
generated by $\widehat{c_{\epsilon}}$, which sends
$$t_1\to \epsilon t_1, \quad t_2\to \epsilon^{-1} t_2, \quad
\Omega_1\to \Omega_1, \quad \Omega_2\to \Omega_2.$$
Let $x_i=t_i^{2n}$ for $i=1,2$. Then $B^G=T_{a^{2n}}/(x_1x_2)$
by identifying $x_i\in B^G$ with $x_i\in T_{a^{2n}}$
and $\Omega_i\in B^G$ with $y_i\in T_{a^{2n}}$ for $i=1,2$.
The quotient group $W:=\widehat{H}/G$ is generated by the involutions
$$\widehat{s}: x_1\to x_2, \quad x_2\to x_1, \quad
\Omega_1\to -a\Omega_2, \quad \Omega_2\to -a^{-1}\Omega_1,$$
and
$$\widehat{d_1}:x_1\to x_1, \quad x_2\to x_2, \quad
\Omega_1\to -\Omega_1, \quad \Omega_2\to -\Omega_2.$$
Note that $\widehat{d_1}=\zeta_1$ and $\widehat{s}=\zeta_2$.
Therefore
$$B^{\Phi(H)}=(B^G)^W=(T_{a^{2n}}/(x_1x_2))^{W}
\cong (T_{a^{2n}})^{W}/(x_1x_2),$$
where the last equation follows from Lemma \ref{zzlem5.4}(2).
By Example \ref{zzex5.9}(iii) $(T_{a^{2n}})^{W}$ is a cci,
whence a gci by Lemma \ref{zzlem5.3}(1).
Therefore $B^{\Phi(H)}$ is a gci, and $A^{Q_7}$ is a gci.\\
~\\
{\bf Case (iv): $H=Q_8$.}\\
 The argument is similar, but a bit more
complicated.

Recall that $Q_8\subset U$ is generated $s$ and $c_{\epsilon,-}$,
where $\epsilon$ is a primitive $4n$th root of unity. We again use the
filtration given in Lemma \ref{zzlem7.2}(3), and let $B$ be the
associated graded algebra.
Let $G$ be the subgroup of $\widehat{H}:=\Phi(H)$ generated by
$\widehat{c_{\epsilon^2}}=\widehat{c_{\epsilon,-}^2}$, which sends
$$t_1\to \epsilon^2 t_1, \quad t_2\to \epsilon^{-2} t_2, \quad
\Omega_1\to \Omega_1, \quad \Omega_2\to \Omega_2.$$
Let $x_i=t_i^{2n}$ for $i=1,2$. Then $B^G=T_{a^{2n}}/(x_1x_2)$
by identifying $x_i\in B^G$ with $x_i\in T_{a^{2n}}$
and $\Omega_i\in B^G$ with $y_i\in T_{a^{2n}}$ for $i=1,2$.
The quotient group $W:=\widehat{H}/G$ is generated by the involutions
$$\phi_1:=\widehat{s}: x_1\to x_2, \quad x_2\to x_1, \quad
\Omega_1\to -a\Omega_2, \quad \Omega_2\to -a^{-1}\Omega_1,$$
and
$$\phi_2:=\widehat{c_{\epsilon,-}}:x_1\to -x_1, \quad x_2\to -x_2, \quad
\Omega_1\to -\Omega_1, \quad \Omega_2\to -\Omega_2.$$ Hence both $\phi_1$ and
$\phi_2$ preserve the relation $x_1x_2=0$. To show $A^{Q_8}$ is a
gci, it suffices to show that $B^{\Phi(Q_8)}$ is a gci, or,
equivalently, that $(B^G)^W=(T_{a^{2n}}/(x_1x_2))^W$ is. By Lemma
\ref{zzlem5.4}(3), we need to show only that $T_{a^{2n}}^W$ is a gci. Set
$q=a^{2n}$. We need to show that $T_q^{W}$ is a gci. Note that
$\phi_1=\zeta_1$ in Example \ref{zzex5.9}. By Example \ref{zzex5.9}
case (ii), $R:=T_q^{\zeta_2}$ is isomorphic to $C/(\Theta)$ or
$C/(\Theta_1,\Theta_2)$. The proofs for these two cases are similar,
so let us assume that $R\cong C/(\Theta_1,\Theta_2)$, where $C$ and
$\Theta_i$ are given in Example \ref{zzex5.9}. Now $\phi_2$ induces
an automorphism of $R$ which is equal to
$$\eta_2: X_1\to -X_1, \quad X_2\to X_2, \quad Y_1\to -Y_1, \quad Y_2\to Y_2, \quad Z_1\to
Z_1, \quad Z_2\to Z_2$$ where $X_1,X_2,Y_1,Y_2, Z_1,Z_2$ are the
generators of $R$ and $C$. Further, $\eta_2$ preserves $\Theta_1$ and
$\Theta_2$. By Lemma \ref{zzlem5.4}(3), it is enough to show that
$C^{\eta_2}$ is a gci. One can check that fact directly, or use the
filtration to reduce to a simpler case. To illustrate the argument, let us
use the filtration. We need to return to the algebra $C$ defined in
Example \ref{zzex5.9}. One can check easily that $X_2$ is a regular
normal element. Define
$$F(n)= \begin{cases} x_2^n C, & n\geq 0\\ C & n\geq 0\end{cases}.$$
Then $\eta_2$ preserves the filtration. By Proposition
\ref{zzpro0.4}(1), it suffices to show that $(gr_F C)^{\eta_2}$ is
a gci. The filtration $F$ gives us $\gr_F C=C/(X_2)[X_2;\sigma]$,
for some automorphism $\sigma$. By Lemma
\ref{zzlem5.4}(1), it is enough to show that $(C/(X_2))^{\eta_2}$ is
a gci, where we re-use $\eta_2$ for the automorphism of $C/(X_2)$
induced by $\eta_2$. We need to return to the algebra in Example
\ref{zzex5.9} to see that $Y_2$ is a regular normal element in
$C/(X_2)$. The same argument shows that we need to show that
$(C/(X_2,Y_2))^{\eta_2}$ is a gci. By repeating this argument, the assertion
eventually follows from the fact that
$(C/(X_2,Y_2,Z_1,Z_2))^{\eta_2}$ is a gci (which is isomorphic to
the second Veronese subring of $k[X_1,Y_1]$). This finishes the case
when $\alpha\neq -2$.\\
~\\
Next we consider the case when $\alpha=-2$, and $A=A(-2,-1)$. In this case, $a=b=-1$, and we
let $\Omega=du+ud$, and  use the filtration defined in Lemma
\ref{zzlem7.2}(2), which is $H$-stable for each of the four groups.
The associated graded ring $B:=\gr_F A$ is isomorphic to $k_{-1}[t_1,t_2,t_3]$, where $t_1:=\widehat{u}$,
$t_2:=\widehat{d}$ and $t_3:=\widehat{\Omega}$. By Proposition
\ref{zzpro0.4}(1), we need to show only that
$B^{\Phi(H)}$ is a gci. Again we have four groups to consider.\\
~\\
{\bf Case (i): $H= Q_5 = BD_{4n}$.}\\
The group $Q_5$ is generated by $s_1$ and
$c_{\epsilon}$, where $\epsilon$ is a primitive $2n$th root of unity.
Let $G$ be the subgroup of $\widehat{H}:=\Phi(H)$ generated by
$\widehat{c_{\epsilon}}$, which sends
$$t_1\to \epsilon t_1, \quad t_2\to \epsilon^{-1} t_2, \quad t_3\to t_3.$$
Then $B^G$ is isomorphic to
$$D:=k[x_1,x_2, x_2]/(x_1x_2-(-1)^{{2n \choose 2}}x_3^{2n})\otimes k[t_3],$$
where $x_1=t_1^{2n}$, $x_2=t_2^{2n}$ and $x_3=t_1t_2$. The quotient
group $\widehat{H}/G$ is generated by the involution $\widehat{s_1}$,
which sends
$$x_1\to x_2, \quad x_2\to x_1, \quad x_3\to x_3, \quad t_3\to -t_3.$$
Now $h:=\widehat{s_1}$ satisfies the hypotheses of Lemma \ref{zzlem5.5}.
Therefore $D^h$ is a gci, and since $B^{\widehat{H}}=(B^G)^h=D^h$, we have $B^{\widehat{H}}$, and hence $A^H$, is a gci.\\
~\\
{\bf Case (ii): $H=Q_6$.}\\
 The first part of argument is similar to the
argument for $H=Q_5$ (replacing $2n$ by $n$). Let
$B=\gr_F A=(k_{-1}[t_1,t_2])[t_3,\sigma]$. Then every $\phi\in H$
preserves $\Omega$, we have $B^{\Phi(H)}=( k_{-1}[t_,t_2]^H)[t_3,\sigma]$.
Since $H$ has trivial homological determinant when acting on
$k_{-1}[t_,t_2]$, Theorem \ref{zzthm0.1} says that
$k_{-1}[t_,t_2]^H$ is a cci. Thus $B^{\Phi(H)}=(
k_{-1}[t_,t_2]^H)[t_3,\sigma]$ is a cci, whence a gci, and $A^H$ is a gci.\\
~\\
{\bf Case (iii):  $H=Q_7$.}\\
Recall that $Q_7\subset U$ is generated by
$d_1$, $s$ and $c_{\epsilon}$, where $\epsilon$ is a primitive $2n$th
root of unity. Let $B=\gr_F A=k_{-1}[t_1,t_2,t_3]$ as above. Let $G$
be the subgroup of $\Phi(H)$ generated by $\widehat{c_{\epsilon}}$.
Then $B^G$ is isomorphic to
$$D:=k[x_1,x_2, x_2]/(x_1x_2-(-1)^{{2n \choose 2}}x_3^{2n})\otimes k[t_3],$$
where $x_1=t_1^{2n}$, $x_2=t_2^{2n}$ and $x_3=t_1t_2$. The quotient
group $W:=\widehat{H}/G$ is generated by the involutions
$$\phi_1:=\widehat{s}: x_1\to x_2, \quad x_2\to x_1, \quad x_3\to -x_3, 
\quad t_3\to t_3,$$
and
$$\phi_2:=\widehat{d_1}: x_1\to x_1, \quad x_2\to x_2, \quad x_3\to -x_3, 
\quad t_3\to -t_3.$$
Since both $\phi_1$ and $\phi_2$ preserve the relation
$f=x_1x_2-(-1)^{{2n \choose 2}}x_3^{2n}$, we can lift $\phi_1$ and
$\phi_2$ to the level of $D_1:=k[x_1,x_2,x_3,t_3]$. By Lemma
\ref{zzlem5.4}(3), it suffices to show that $D_1^W$ is a gci. Now
$D_1$ is commutative, $W$ is a group of order 4, so it is an easy
computation to show that $D_1^W$ is a cci (see also \cite{G2, Na}), 
whence a gci .\\
~\\
{\bf Case (iv): $H=Q_8$.}\\
 Recall that $Q_8\subset U$ is generated by $s$
and $c_{\epsilon,-}$, where $\epsilon$ is a primitive $4n$th root of
unity. Let $B=\gr_F A=k_{-1}[t_1,t_2,t_3]$ where $t_1:=\widehat{u}$,
$t_2:=\widehat{d}$ and $t_3:=\widehat{\Omega}$. Let $G$ be the
subgroup of $\Phi(H)$ generated by $\widehat{c_{\epsilon^2}}$. Then
$B^G$ is isomorphic to
$$D:=k[x_1,x_2, x_2]/(x_1x_2-(-1)^{{2n \choose 2}}x_3^{2n})\otimes k[t_3],$$
where $x_1:=t_1^{2n}$, $x_2:=t_2^{2n}$ and $x_3:=t_1t_2$. The quotient
group $W:=\widehat{H}/G$ is generated by the involutions
$$\phi_1:=\widehat{s}: x_1\to x_2, \quad x_2\to x_1, \quad x_3\to -x_3, 
\quad t_3\to t_3$$
and
$$\phi_2:=\widehat{c_{\epsilon,-}}: x_1\to -x_1, \quad x_2\to -x_2, 
\quad x_3\to -x_3, \quad t_3\to -t_3.$$
Both $\phi_1$ and $\phi_2$ preserves the relation
$f:=x_1x_2-(-1)^{{2n \choose 2}}x_3^{2n}$.  We can lift $\phi_1$ and
$\phi_2$ to the level of $D_1:=k[x_,x_2,x_3,t_3]$. By Lemma
\ref{zzlem5.4}(3), it suffices to show that $D_1^W$ is a gci. Now
$D_1$ is commutative, $W$ is a group of order 4, and it is an easy
computation to show that $D_1^W$ is a cci (see also \cite{G2, Na}), 
whence a gci. This case finishes the
argument when $\alpha=-2$.

Combining all the cases above, the proof is now complete.
\end{proof}

\section{Proof of Theorem \ref{zzthm0.3}}
\label{zzsec9}

We are now ready to prove Theorem \ref{zzthm0.3}.

\begin{proof}[Proof of Theorem \ref{zzthm0.3}] 
Recall that $A=A(\alpha,\beta)$ is a noetherian
graded down-up algebra, and $H$ is a finite subgroup of $\Aut(A)$. By
Lemma \ref{zzlem5.3}(2), if $A^H$ is a gci, then $A^H$ is a cyclotomic 
Gorenstein. To prove
Theorem \ref{zzthm0.3} one needs to show that, if $A^H$ is
cyclotomic Gorenstein, then $A^H$ is a gci and $H$ is generated by
bireflections. This statement is equivalent to the following claim.
\begin{enumerate}
\item[Claim 1:]
{\it For every $H$, either $A^H$ is not cyclotomic Gorenstein, or
both $A^H$ is a gci and $H$ is generated by bireflections.}
\end{enumerate}

If $H=Q_3$, or is conjugate to $Q_3$, and $n>1$, then by Lemma 
\ref{zzlem8.4} there
is an $A$ such that $Q_3\subset \Aut(A)$ and $A^{Q_3}$ is not
cyclotomic Gorenstein. By Lemma \ref{zzlem7.3}(4), if $Q_3\subset \Aut(B)$
for any other noetherian graded down-up algebra $B$, then $B^{Q_3}$
is not cyclotomic Gorenstein. By the same reasoning, $B^{Q_3}$ is cyclotomic
Gorenstein when $n=1$, and $Q_3$ is generated by bireflections. 
Similarly, if $Q_4\subset \Aut(A)$ for
any noetherian graded down-up algebra $A$, then $A^{Q_4}$ is not
cyclotomic Gorenstein [Lemma \ref{zzlem8.7}].

Now Claim 1 is equivalent to the following claim.
\begin{enumerate}
\item[Claim 2:]
{\it For every $H$, if $H$ is not conjugate to $Q_3$ or $Q_4$, then
$A^H$ is a gci and $H$ is generated by bireflections.}
\end{enumerate}

We will analyze $A^H$ dependent on $\Aut(A)$, which is dependent on
the parameters $(\alpha, \beta)$. First, we present a lemma that 
deals with some subgroups of
$SL_2(k)$.

\begin{lemma}
\label{zzlem9.1} Let $\Omega =du-aud$ where $a\neq -1$, and let $F$
be the filtration defined in Lemma {\rm{\ref{zzlem7.2}(2)}}. Let $H\subset
SL_2(k)$ be a finite subgroup. Suppose $g(\Omega)=\Omega$ for all
$g\in H$. Then $A^H$ is a gci.
\end{lemma}

\begin{proof} Since $g(\Omega)=\Omega$ for all $g\in H$, the filtration
given in \ref{zzlem7.2}(2) is $H$-stable. Since $B:=\gr_F A$ is a
skew polynomial ring $k_{p_{ij}}[t_1,t_2,t_3]$ that is generated by
$t_1=\widehat{d}, t_2=\widehat{u}$ and $t_3=\widehat{\Omega}$, the
$\widehat{H}$-action on $B$ fixes $t_3$. By Lemma
\ref{zzlem5.4}(1), it suffices to that $k_a[t_1,t_2]^H$ is a gci.
Since $H\subset SL_2(k)$ and since $a\neq -1$, $H\subset
\AutSL(k_a[t_1,t_2])$. By \cite[Theorem 3.3]{JoZ}, $k_a[t_1,t_2]^H$
is AS Gorenstein. By Theorem \ref{zzthm0.1}, $k_a[t_1,t_2]^H$ is a
gci, and therefore the assertion follows. 
\end{proof}

We note that Lemma \ref{zzlem9.1} applies to
all values of $\alpha$ and $\beta$, except in the case that all roots
of the ``character polynomial" are $-1$, which happens when
$$x^2-\alpha x  -\beta = x^2 +2x +1,$$
hence when $\alpha = -2$ and $\beta = -1$.  The argument in the proof of Lemma \ref{zzlem9.1}
also applies to the case when
$H \subset O \cap SL_2(k)$ and $a= -1$.\\

We now use the results we have assembled to prove Theorem \ref{zzthm0.3}.\\
~\\
{\bf Case 1: $(\alpha, \beta)=(2,-1)$ or $(0,1)$.}
In these two case  $\Aut(A)=GL_2(k)$ by Lemma \ref{zzlem7.2}(2).  In
the first case the roots of the character equation are $a=b=1$, and in the second case they are $a=1$ and $b=-1$\\

Let $\Omega=du-ud$. Then $g(\Omega)=(\det g)\Omega$ for all $g\in
GL_2(k)$. If $H$ is a subgroup of $SL_2(k)$,
Lemma \ref{zzlem9.1} shows that $A^H$ is always a gci, and Lemma \ref{zzlem7.3}(3)
shows that all elements ($\neq \mathbb{I}$) of $H$ are bireflections, so
Claim 2 follows in this case. If
$H$ is not a subgroup of $SL_2(k)$, by Lemma \ref{zzlem7.3}(2), $H$
satisfies the exact sequence \eqref{E3.0.1}. If $G:=H\cap SL_2(k)$
is cyclic, then $H$ is isomorphic to a group given in $(A_{n,1})$,
$(A_{n,2})$ or $(A_{n,5})$. (Note that the groups occurring in
$(A_{n,3})$ and $(A_{n,4})$ are $Q_3$ and $Q_4$, respectively, so have been
eliminated). For any of the cases $(A_{n,i})$ for $i=1,2,5$, the
proof is very similar, so we do only one case, say $(A_{n,5})$. The
filtration given in Lemma \ref{zzlem7.2}(2) is $H$-stable for $\Omega = du-ud$
(since $a=1$ is a root of the character equation in both cases, and $g(\Omega)=(\det g)\Omega$). By
Proposition \ref{zzpro0.4}(1), it suffices to show that
$(\gr_F A)^{\Phi(H)}$ is a gci. Now  $\gr_F A=(k[t_1,t_2])[t_3;\sigma]$
and $(\gr_F A)^{\Phi(G)}=(k[t_1,t_2]^G)[t_3;\sigma]$. Note that  $k[t_1,t_2]^G$
is the algebra $S$ given in the discussion of case $(A_{n,5})$. Let
$g\in H\setminus G$, and let $h$ be the induced action of $g$ on $S$.
Then $h$ extends to an action on $S[t_3;\sigma]$ by $h(t_3)=-t_3$.
One can check that $\widehat{g}=h$.  By Lemma \ref{zzlem3.3}, the
hypotheses in Lemma \ref{zzlem5.5} hold. Hence, by Lemma
\ref{zzlem5.5}, $(S[t_3;\sigma])^h$ is a gci. Finally,
$$(\gr_F
A)^{\Phi(H)}=((\gr_F A)^{\Phi(G)})^h=(S[t_3;\sigma])^h,$$ which is
now a gci. Combining the above, we have that $A^H$ is a gci when $H$
is in $(A_{n,5})$.

A similar argument, using Lemmas \ref{zzlem3.3} and \ref{zzlem5.5},
 works when $G$ is of type $D_n$, $E_6$, $E_7$, or
$E_8$, finishing Case 1.\\
~\\
{\bf Case 2: $H\subset O$.} Then $H$ is listed as in Lemma \ref{zzlem3.4}.\\

Since $H$ is not conjugate to $Q_3$ and $Q_4$. One needs to
consider only $Q_1$ and $Q_2$. An argument as in the proof of
Lemma \ref{zzlem9.1} establishes the result for $Q_1$, and an argument as in the proof of
case $(A_{n,5})$ establishes the theorem for $Q_2$.

If $\beta\neq -1$, then $\Aut(A)=O$, and hence the theorem is established
when $\beta\neq -1$.\\
~\\
{\bf Case 3: $\beta=-1$ and $H\subset U$.}\\
By Lemma \ref{zzlem3.5}, $H$ is
isomorphic to $Q_5$, $Q_6$, $Q_7$ or $Q_8$. The case where $\alpha=2$ is
covered in Case 1.  The case where $\alpha\neq  2$ is covered in
Proposition 8.8.\\

Hence the proof is complete in all cases.
\end{proof}

\section*{Acknowledgments}
The authors thank Michael Mossinghoff from Davidson College for the
proof of Lemma \ref{zzlem8.6}.  Wake Forest University undergraduate
Brian Smith computed several examples that were useful in this
work. E. Kirkman was partially supported by grant \#208314 from 
the Simons Foundation. J.J. Zhang was partially supported by the 
US National Science Foundation: NSF grant DMS 0855743.  

\end{document}